\documentclass[a4paper]{article}
\usepackage{hyperref}
\usepackage{stmaryrd}
\usepackage{amsmath,bm}      
\usepackage{adjustbox}
\usepackage{amsfonts}
\usepackage{amsthm}
\usepackage{amssymb}
\usepackage{mathrsfs}
\DeclareSymbolFontAlphabet{\mathrsfs}{rsfs}
\usepackage[mathscr]{eucal}
\usepackage{tikz}
\usepackage{graphicx}
\usepackage{multicol}
\makeatletter
\newcommand*\bigcdot{\mathpalette\bigcdot@{.5}}
\newcommand*\bigcdot@[2]{\mathbin{\vcenter{\hbox{\scalebox{#2}{$\m@th#1\bullet$}}}}}
\makeatother
\usetikzlibrary{matrix, arrows}
\usetikzlibrary{positioning}
\usetikzlibrary{decorations.pathreplacing}
\theoremstyle{plain}
\newtheorem*{Example*}{\bfseries{\emph{Example}}}
\newtheorem*{Notation*}{\bfseries{\emph{Notation}}}
\newtheorem*{Theorem*}{\bfseries{\emph{Theorem}}}
\newtheorem*{Lemma*}{\bfseries{\emph{Lemma}}}
\newtheorem*{Proposition*}{\bfseries{\emph{Proposition}}}
\newtheorem*{Corollary*}{\bfseries{\emph{Corollary}}}
\newtheorem*{Remark*}{\bfseries{\emph{Remark}}}
\newtheorem*{Remarks*}{\bfseries{\emph{Remarks}}}
\newtheorem*{Def*}{\bfseries{\emph{Definition}}}
\newtheorem*{Conjecture*}{\bfseries{\emph{Conjecture}}}
\newtheorem*{sketch proof*}{Sketch proof}
\newtheorem{Theorem}{\bfseries{\emph{Theorem}}}[section]

\newtheorem{Remark}[Theorem]{\bfseries{\emph{Remark}}}

\newtheorem{Def}[Theorem]{\bfseries{\emph{Definition}}}

\newcommand{\cris}{\text{cris}}

\newcommand{\integral}{\text{int}}
\newcommand{\fractional}{\text{frac}}

  \title{Higher displays arising from filtered de Rham-Witt complexes}
  \date{February 15, 2019}
  \author{Oli Gregory\footnote{The first named author is supported by the ERC Consolidator Grant 681838 “K3CRYSTAL”.} \ and Andreas Langer}
\begin{document}
\maketitle
\begin{abstract}
 For a smooth projective scheme $X$ over a ring $R$ on which $p$ is nilpotent that meets some general assumptions we prove that the crystalline cohomology is equipped with the structure of a higher display which is a relative version of Fontaine's strongly divisible lattices. Frobenius-divisibility is induced by the Nygaard filtration on the relative de Rham-Witt complex. For a nilpotent PD-thickening $S/R$ we also consider the associated relative display and can describe it explicitly by a relative version of the Nygaard filtration on the de Rham-Witt complex associated to a lifting of $X$ over $S$. We prove that there is a crystal of relative displays if moreover the mod $p$ reduction of $X$ has a smooth and versal deformation space. 
\end{abstract}
\section{Introduction}
For a ring $R$ in which $p$ is nilpotent, we constructed in \cite{LZ07} an exact tensor category of displays which contains the displays associated to $p$-divisible groups \cite{Zin02} as a full subcategory. If $R=k$ is a perfect field, a display is a finitely generated free $W(k)$-module $M$ endowed with an injective Frobenius-linear map $F:M\rightarrow M$. In general, displays can be regarded as a relative version of Fontaine's strongly divisible lattices \cite{Fon83}. For a smooth projective scheme $X$ over $\text{Spec }R$ we developed a strategy in \cite{LZ07} to equip the crystalline cohomology $H_{\cris}^{n}(X/W(R))$ with the structure of a display. For a precise statement see below. The following assumptions were essential in the construction of such ``geometric'' displays:
\newline
\par
There exists a compatible system of smooth liftings $X_{n}/W_{n}(R)$ for $n\in\mathbb{N}$ of $X/R$ such that the following properties hold:
\newline
\par
(A1) The cohomology groups $H^{i}(X_{n},\Omega_{X_{n}/W_{n}(R)}^{j})$ are for each $n, i$ and $j$ locally free $W_{n}(R)$-modules of finite type.
\newline
\par 
(A2) For each $n$ the de Rham spectral sequence degenerates at $E_{1}$
\begin{equation*}
E_{1}^{i,j}=H^{j}(X_{n},\Omega_{X_{n}/W_{n}(R)}^{i})\Rightarrow\mathbb{H}^{i+j}(X_{n},\Omega_{X_{n}/W_{n}(R)}^{\bullet})
\end{equation*} 
(compare ($\ast$) and ($\ast\ast$) in (\cite{LZ07}, p.150) and assumptions 5.2, 5.3 in \cite{LZ07}). For example, these assumptions are satisfied by K3 surfaces, abelian schemes and smooth relative complete intersections (see \cite{LZ07}, Introduction).
\newline
\par
Let $I_{R}:=VW(R)$ and $W\Omega_{X/R}^{\bullet}$ be the relative de Rham-Witt complex as constructed in \cite{LZ04}. For $r\geq 0$ define the complex $\mathcal{N}^{r}W\Omega_{X/R}^{\bullet}$ as follows:
\begin{equation*}
(W\Omega_{X/R}^{0})_{[F]}\xrightarrow{d}(W\Omega_{X/R}^{1})_{[F]}\xrightarrow{d}\cdots\xrightarrow{d}(W\Omega_{X/R}^{r-1})_{[F]}\xrightarrow{dV}W\Omega_{X/R}^{r}\xrightarrow{d}\cdots
\end{equation*}
This is a complex of $W(R)$-modules where $(W\Omega_{X/R}^{i})_{[F]}$, for $i<r$, is considered as a $W(R)$-module via restriction of scalars along $W(R)\xrightarrow{F}W(R)$. It was conjectured in (\cite{LZ07}, Conj. 5.8) that the predisplay structure (Appendix, Def. \ref{predisplay}) on $P_{0}:=H_{\cris}^{n}(X/W(R))$, defined by the data $P_{r}:=\mathbb{H}^{n}(X,\mathcal{N}^{r}W\Omega_{X/R}^{\bullet})$ and maps $\hat{\alpha}_{r}:I_{R}\otimes P_{r}\rightarrow P_{r+1}$, $\hat{\iota}_{r}:P_{r+1}\rightarrow P_{r}$ and $\hat{F}_{r}:P_{r}\rightarrow P_{0}$ induced by the corresponding maps of complexes $\hat{\alpha}_{r}:I_{R}\otimes\mathcal{N}^{r}W\Omega_{X/R}^{\bullet}\rightarrow\mathcal{N}^{r+1}W\Omega_{X/R}^{\bullet}$, $\hat{\iota}_{r}:\mathcal{N}^{r+1}W\Omega_{X/R}^{\bullet}\rightarrow\mathcal{N}^{r}W\Omega_{X/R}^{\bullet}$ and $\hat{F}_{r}:\mathcal{N}^{r}W\Omega_{X/R}^{\bullet}\rightarrow W\Omega_{X/R}^{\bullet}$ given in (\cite{LZ07}, (5)) and the diagram below  and in the same order between the vertically written complexes
\begin{equation*}
\begin{tikzpicture}[descr/.style={fill=white,inner sep=1.5pt}]
        \matrix (m) [
            matrix of math nodes,
            row sep=2.5em,
            column sep=2em,
            text height=1.5ex, text depth=0.25ex
        ]
        {  I_{S}\otimes(W\Omega_{X/R}^{0})_{[F]} & (W\Omega_{X/R}^{0})_{[F]} & (W\Omega_{X/R}^{0})_{[F]} & W\Omega_{X/R}^{0}  \\
            \cdots &\cdots & \cdots & \cdots \\
            I_{S}\otimes(W\Omega_{X/R}^{r-1})_{[F]} & (W\Omega_{X/R}^{r-1})_{[F]} & (W\Omega_{X/R}^{r-1})_{[F]} & W\Omega_{X/R}^{r-1} \\
            I_{S}\otimes W\Omega_{X/R}^{r} & (W\Omega_{X/R}^{r})_{[F]} & W\Omega_{X/R}^{r} & W\Omega_{X/R}^{r} \\
            I_{S}\otimes W\Omega_{X/R}^{r+1} & W\Omega_{X/R}^{r+1} & W\Omega_{X/R}^{r+1} & W\Omega_{X/R}^{r+1} \\
            I_{S}\otimes W\Omega_{X/R}^{r+2} & W\Omega_{X/R}^{r+2} & W\Omega_{X/R}^{r+2} & W\Omega_{X/R}^{r+2} \\ 
            \cdots &\cdots & \cdots & \cdots \\ };

        \path[overlay,->, font=\scriptsize]
        (m-1-1) edge (m-1-2)
        (m-1-2) edge node [above] {$p$} (m-1-3)
        (m-1-3) edge node [above] {$\mathrm{id}$} (m-1-4)
        (m-3-1) edge (m-3-2)
        (m-3-2) edge  node [above] {$p$} (m-3-3)
        (m-3-3) edge  node [above] {$\mathrm{id}$} (m-3-4)
        (m-4-1) edge node [above] {$\tilde{F}$} (m-4-2)
        (m-4-2) edge  node [above] {$V$} (m-4-3)
        (m-4-3) edge  node [above] {$F$} (m-4-4)
        (m-5-1) edge  node [above] {$mult$} (m-5-2)
        (m-5-2) edge  node [above] {$\mathrm{id}$} (m-5-3)
        (m-5-3) edge  node [above] {$pF$} (m-5-4)
        (m-6-1) edge node [above] {$mult$} (m-6-2)
        (m-6-2) edge node [above] {$\mathrm{id}$} (m-6-3)
        (m-6-3) edge node [above] {$p^{2}F$} (m-6-4)
        (m-1-1) edge node [left] {$\mathrm{id}\otimes d$}(m-2-1)
        (m-1-2) edge node [left] {$d$}(m-2-2)
        (m-1-3) edge node [left] {$d$}(m-2-3)
        (m-1-4) edge node [left] {$d$}(m-2-4)
        (m-2-1) edge node [left] {$\mathrm{id}\otimes d$}(m-3-1)
        (m-2-2) edge node [left] {$d$}(m-3-2)
        (m-2-3) edge node [left] {$d$}(m-3-3)
        (m-2-4) edge node [left] {$d$}(m-3-4)
        (m-3-1) edge node [left] {$\mathrm{id}\otimes dV$}(m-4-1)
        (m-3-2) edge node [left] {$d$}(m-4-2)
        (m-3-3) edge node [left] {$dV$}(m-4-3)
        (m-3-4) edge node [left] {$d$}(m-4-4)
        (m-4-1) edge node [left] {$\mathrm{id}\otimes d$}(m-5-1)
        (m-4-2) edge node [left] {$dV$}(m-5-2)
        (m-4-3) edge node [left] {$d$}(m-5-3)
        (m-4-4) edge node [left] {$d$}(m-5-4)
        (m-5-1) edge node [left] {$\mathrm{id}\otimes d$}(m-6-1)
        (m-5-2) edge node [left] {$d$}(m-6-2)
        (m-5-3) edge node [left] {$d$}(m-6-3)
        (m-5-4) edge node [left] {$d$}(m-6-4)
        (m-6-1) edge node [left] {$\mathrm{id}\otimes d$}(m-7-1)
        (m-6-2) edge node [left] {$d$}(m-7-2)
        (m-6-3) edge node [left] {$d$}(m-7-3)
        (m-6-4) edge node [left] {$d$}(m-7-4);
        
\end{tikzpicture}
\end{equation*}
define a display structure on $P_{0}$ (Appendix, Def. \ref{display}). (Note that in the definition of $\hat{\alpha}_{r}$, $\tilde{F}(V\xi\otimes\omega):=\xi F\omega$).
\par
In this note we prove this conjecture under the assumption $r\leq n<p$, which is a standard hypothesis in integral $p$-adic Hodge theory, and prove the following:
\newpage
\begin{Theorem}\label{Display theorem}
\ 
\par
(a) Let $R$ be a local ring in which $p$ is nilpotent, and let $X$ be a smooth projective scheme over $\text{Spec }R$. Assume that there exists a compatible system of liftings $X_{n}/\text{Spec }W_{n}(R)$ satisfying (A1) and (A2). Then the data $(P_{r},\hat{\iota}_{r},\hat{\alpha}_{r},\hat{F}_{r})$ form a display structure $\mathcal{P}_{R}$ on $H_{\cris}^{n}(X/W(R))$ for $r\leq n<p$.
\par
(b) Assume that there exists in addition a frame $A\rightarrow R$ and a smooth projective $pA$-adic lifting $\mathcal{Y}/\text{Spf }A$ of $X$ such that the $\mathcal{Y}_{n}:=\mathcal{Y}\times_{\text{Spf }A}\text{Spf }A/p^{n}$ satisfy the analogous assumptions (A1) and (A2). Then the display structure on $H_{\cris}^{n}(X/W(R))$ obtained by base change from the window structure on $H_{\cris}^{n}(X/A)$ (compare \cite{LZ07} Theorem 5.5 and Corollary 5.6) is isomorphic to $\mathcal{P}_{R}$.
\end{Theorem}
\begin{Remark*}
Theorem \ref{Display theorem}(a) was shown for reduced rings $R$ in (\cite{LZ07}, Theorem 5.7).
\end{Remark*}
In our second main result we derive a relative version of Theorem \ref{Display theorem} on relative displays using a modified version of the complexes $\mathcal{N}^{r}W\Omega_{X/R}^{\bullet}$. For this, let $S\rightarrow R$ be a homomorphism of rings in which $p$ is nilpotent and such that the kernel $\mathfrak{a}$ is equipped with nilpotent divided powers. In \cite{LZ19} and \cite{Gre17} we considered the Witt frames $\mathcal{W}_{S}$, $\mathcal{W}_{S/R}$, $\mathcal{W}_{R}$ (Appendix, Def. \ref{Witt frame}, \ref{relative Witt frame}) and the canonical homomorphisms of frames $\mathcal{W}_{S}\xrightarrow{u}\mathcal{W}_{S/R}$ and $\mathcal{W}_{S/R}\rightarrow\mathcal{W}_{R}$. Let $X/\text{Spec }R$ be as before and let $X_{S}$ be a smooth projective lifting of $X$ over $\text{Spec }S$, admitting liftings $(X_{S})_{n}$ over $\text{Spec }W_{n}(S)$ that satisfy the assumptions (A1) and (A2). Let $\tilde{\mathfrak{a}}$ be the logarithmic Teichm\"{u}ller ideal in $W(S)$, as defined in (\cite{Zin02}, 1.4 and Appendix, Def. \ref{relative Witt frame}). Then $\mathcal{J}:=\tilde{\mathfrak{a}}\oplus VW(S)$ is the kernel of the composite map $W(S)\rightarrow S\rightarrow R$ and is again equipped with a PD-structure (see \cite{Zin02}, 2.3). Define the complex $\mathcal{N}_{rel/R}^{r}W\Omega_{X_{S}/S}^{\bullet}$ as follows:
\resizebox{1.0\linewidth}{!}{
  \begin{minipage}{\linewidth}
\begin{align*}
(W\mathcal{O}_{X_{S}})_{[F]}\oplus\tilde{\mathfrak{a}}^{r}W\mathcal{O}_{X_{S}}\xrightarrow{d\oplus d}
& (W\Omega_{X_{S}/S}^{1})_{[F]}\oplus\tilde{\mathfrak{a}}^{r-1}W\Omega_{X_{S}/S}^{1}\xrightarrow{d\oplus d}\cdots \\
& \cdots\xrightarrow{d\oplus d}(W\Omega_{X_{S}/S}^{r-1})_{[F]}\oplus\tilde{\mathfrak{a}}W\Omega_{X_{S}/S}^{r-1}\xrightarrow{dV+d}W\Omega_{X_{S}/S}^{r}\xrightarrow{d}\cdots
\end{align*}
\end{minipage}}
The maps $\hat{\alpha}_{r}$, $\hat{\iota}_{r}$, $\hat{F}_{r}$ on $\mathcal{N}^{r}W\Omega_{X_{S}/S}^{\bullet}$ that define the predisplay structure on $H_{\cris}^{n}(X/W(S))=H_{\cris}^{n}(X_{S}/W(S))$ can easily be extended to maps on the complexes $\mathcal{N}_{rel/R}^{r}W\Omega_{X_{S}/S}^{\bullet}$, where multiplication by $p$ on $W\Omega_{X_{S}/S}^{j}$ is replaced by the map
\begin{equation*}
\pi:W\Omega_{X_{S}/S}^{j}\oplus\tilde{\mathfrak{a}}^{r-j}W\Omega_{X_{S}/S}^{j}\rightarrow W\Omega_{X_{S}/S}^{j}\oplus\tilde{\mathfrak{a}}^{r-j-1}W\Omega_{X_{S}/S}^{j}
\end{equation*}
which is multiplication by $p$ on $W\Omega_{X_{S}/S}^{j}$ and the inclusion on the other summand. The divided Frobenius $\hat{F}_{r}$ is defined on the subcomplex $\mathcal{N}^{r}W\Omega_{X_{S}/S}^{\bullet}$ as before and on $\tilde{\mathfrak{a}}^{r-j}W\Omega_{X_{S}/S}^{j}$ it is defined to be the zero map. In analogy to (\cite{LZ07}, (5)) we get induced maps of complexes
\begin{align*}
& \hat{\alpha}_{r}:\mathcal{J}\otimes\mathcal{N}_{rel/R}^{r}W\Omega_{X_{S}/S}^{\bullet}\rightarrow\mathcal{N}_{rel/R}^{r+1}W\Omega_{X_{S}/S}^{\bullet} \\
& \hat{\iota}_{r}:\mathcal{N}_{rel/R}^{r+1}W\Omega_{X_{S}/S}^{\bullet}\rightarrow\mathcal{N}_{rel/R}^{r}W\Omega_{X_{S}/S}^{\bullet} \\
& \hat{F}_{r}:\mathcal{N}_{rel/R}^{r}W\Omega_{X_{S}/S}^{\bullet}\rightarrow W\Omega_{X_{S}/S}^{\bullet}
\end{align*}
It is then easy to see that the above data form a predisplay $\mathcal{P}_{S/R}$ on the hypercohomology of these complexes over the relative Witt frame $\mathcal{W}_{S/R}$ (see (\cite{LZ19}, Def. 2 and Appendix, Def. \ref{predisplay}). Then one has:
\begin{Theorem}\label{Relative displays theorem}
\ 
\par
(a) Let $\mathcal{P}_{S}$ be the display over $S$ constructed in Theorem \ref{Display theorem}(a). Then the associated display $u_{\ast}\mathcal{P}_{S}$ under the homomorphism of frames $\mathcal{W}_{S}\xrightarrow{u}\mathcal{W}_{S/R}$ is isomorphic to $\mathcal{P}_{S/R}$. In particular, the predisplay $\mathcal{P}_{S/R}$ is a display over $\mathcal{W}_{S/R}$ in the sense of (\cite{LZ19}, Def. 3 and Appendix, Def. \ref{display}).
\par
(b) Assume that $R$ is artinian local with perfect residue field $k$. Assume that $X_{0}:=X\times_{R}k$ has a smooth versal formal deformation space and that the assumptions (32) and (33) of \cite{LZ19} (analogous to (A1) and (A2)) are satisfied. Then the display $\mathcal{P}_{S/R}$ only depends - up to isomorphism - on $X$, not on the lifting $X_{S}$. The collection $(\mathcal{P}_{S/R})_{S\rightarrow R}$ where $S\rightarrow R$ are PD-morphisms defines a crystal of relative displays.  
\end{Theorem}
\begin{Remark*}
\
\par
\emph{Relative displays were first considered by Zink \cite{Zin02} and Lau \cite{Lau14} in their classification of $p$-divisible groups. Zink associated to a formal $p$-divisible group $\mathcal{G}$ and a nilpotent PD-thickening $S\rightarrow R$ a relative display $\mathrsfs{D}_{S/R}(\mathcal{G})$ (which he calls a triple in \cite{Zin02}) which is - up to isomorphism - the unique relative display lifting the display associated to $\mathcal{G}$ over $R$ to $S$. Using Dieudonn\'{e} displays, i.e. displays defined over the small Witt ring \cite{Zin01}, Zink extended the classification to all $p$-divisible groups over an artinian local ring with perfect residue field, see also \cite{Lau14} and \cite{Mes07}. Over slightly more general rings, called admissible rings, Lau \cite{Lau14} constructed a unique functor from $p$-divisible groups to crystals of relative displays. The relative display is obtained by base change from a window structure associated to the universal $p$-divisible group over the deformation ring. The proof of Theorem \ref{Relative displays theorem}(b) was inspired by this construction.
}
\end{Remark*}

We assume that the reader is familiar with the basic definitions and properties of the relative de Rham-Witt complex \cite{LZ04} including the comparison to crystalline cohomology and the explicit description of the de Rham-Witt complex of a polynomial algebra. In an appendix we recall the basic definitions of the theory of higher displays including frames, windows and relative displays.
\section{Proof of the theorems}
Theorem \ref{Display theorem}(a) is a consequence of the following result which was conjectured in (\cite{LZ07}, Conj. 4.1) and was recently proved by the second named author in (\cite{Lan18}, Thm. 0.2) under the assumption $r<p$:
\begin{Theorem}\label{Theorem 0.2 of Lan18}
Let $R$ be a ring on which $p$ is nilpotent and let $X/\text{Spec }R$ be smooth projective, and $X_{n}/\text{Spec }W_{n}(R)$ a compatible system of smooth liftings. Let $\mathcal{F}^{r}\Omega_{X_{n}/W_{n}(R)}^{\bullet}$ be the following complex:
\begin{equation*}
\resizebox{1.0\hsize}{!}{$
I_{R}\otimes\mathcal{O}_{X_{n}}\xrightarrow{pd}I_{R}\otimes\Omega_{X_{n}/W_{n}(R)}^{1}\xrightarrow{pd}\cdots\xrightarrow{pd}I_{R}\otimes\Omega_{X_{n}/W_{n}(R)}^{r-1}\xrightarrow{d}\Omega_{X_{n}/W_{n}(R)}^{r}\xrightarrow{d}\cdots
$}
\end{equation*}
where $I_{R}:=VW_{n-1}(R)$, and let $\mathcal{N}^{r}W_{n}\Omega_{X/R}^{\bullet}$ be the Nygaard complex, given as
\begin{equation*}
\resizebox{1.0\hsize}{!}{$
(W_{n-1}\Omega_{X/R}^{0})_{[F]}\xrightarrow{d}(W_{n-1}\Omega_{X/R}^{1})_{[F]}\xrightarrow{d}\cdots\xrightarrow{d}(W_{n-1}\Omega_{X/R}^{r-1})_{[F]}\xrightarrow{dV}W_{n}\Omega_{X/R}^{r}\xrightarrow{d}\cdots
$}
\end{equation*}
Then $\mathcal{F}^{r}\Omega_{X_{n}/W_{n}(R)}^{\bullet}$ and $\mathcal{N}^{r}W_{n}\Omega_{X/R}^{\bullet}$ are isomorphic in the derived category of $W_{n}(R)$-modules for $r<p$.
\end{Theorem}

\subsection{Proof of Theorem \ref{Display theorem}(a)}
Under the assumptions (A1) and (A2) it follows from \cite{LZ07} Propositions 3.2 and the projection formula Proposition 3.1 that one has a degenerating spectral sequence $E_{1}^{i,j}\Rightarrow\mathbb{H}^{i+j}(X_{n},\mathcal{F}^{r}\Omega_{X_{n}/W_{n}(R)}^{\bullet})$ where
\[   E_{1}^{i,j}=\left\{
\begin{array}{ll}
      I_{R}\otimes H^{j}(X_{n},\Omega_{X_{n}/W_{n}(R)}^{i}) & \text{for }i<r \\
      H^{j}(X_{n},\Omega_{X_{n}/W_{n}(R)}^{i}) & \text{for }i\geq r
\end{array} 
\right. \]
It follows from the proof of Theorem 2.1 (\cite{Lan18}, Theorem 0.2) that the isomorphisms  $\mathcal{F}^{r}\Omega_{X_{n}/W_{n}(R)}^{\bullet}\cong\mathcal{N}^{r}W_{n}\Omega_{X/R}^{\bullet}$ are compatible for varying $n$ and yield an isomorphism $\mathcal{F}^{r}\Omega_{X_{\bullet}/W_{\bullet}(R)}^{\bullet}\cong\mathcal{N}^{r}W_{\bullet}\Omega_{X/R}^{\bullet}$ of procomplexes in $D_{\text{pro},\text{Zar}}(X)$. This induces an isomorphism resp. a decomposition
\begin{equation*}
P_{r}=\mathbb{H}^{n}(X,\mathcal{F}^{r}\Omega_{X_{\bullet}/W_{\bullet}(R)}^{\bullet})\cong I_{R}L_{0}\oplus I_{R}L_{1}\oplus\cdots\oplus I_{R}L_{r-1}\oplus L_{r}\oplus\cdots\oplus L_{n}
\end{equation*}
where $L_{i}:=H^{n-i}(X,\Omega_{X_{\bullet}/W_{\bullet}(R)}^{i})$. Since the divided Frobenius $\hat{F}_{r}$ is defined on $\mathbb{H}^{n}(X,\mathcal{F}^{r}\Omega_{X_{\bullet}/W_{\bullet}(R)}^{\bullet})$ via Theorem \ref{Theorem 0.2 of Lan18}, we can define $\Phi_{r}:L_{r}\rightarrow P_{0}$ by $\Phi_{r}:=\hat{F}_{r}|_{L_{r}}$.
\par
To show that $(P_{r},\hat{F}_{r},\hat{\iota}_{r},\hat{\alpha}_{r})$ defines a display on $H_{\cris}^{n}(X/W(R))$ is equivalent to the condition that
\begin{equation*}
\bigoplus_{i=0}^{n}\Phi_{i}:\bigoplus_{i=0}^{n}L_{i}\rightarrow\bigoplus_{i=0}^{n}L_{i}
\end{equation*} 
is a $\sigma$-linear isomorphism, or equivalently that $\det(\oplus_{i=0}^{n}\Phi_{i})\in W(R)^{\times}$. This is reduced by base change to the case that $R=k$ is a perfect field in the same way as in the proof of (\cite{LZ07} Thm. 5.5), and then follows from (\cite{Fon83}, p.91) and (\cite{Kat87}, Prop. 2.5).
\subsection{Proof of Theorem \ref{Relative displays theorem}(a)}
We are going to explicitly construct displays over the relative Witt frames. For this, let $S\rightarrow R$ be a homomorphism of rings in which $p$ is nilpotent and such that the kernel $\mathfrak{a}$ is equipped with divided powers. Then the kernel of $W(S)\rightarrow R$ is $\tilde{\mathfrak{a}}\oplus I_{S}$ where $I_{S}:=VW(S)$ and $\tilde{\mathfrak{a}}$ is the logarithmic Teichm\"{u}ller ideal. Then we consider the relative Witt frame $\mathcal{W}_{S/R}$ as defined in \cite{LZ19}.
\par 
We first prove a relative version of Theorem \ref{Theorem 0.2 of Lan18}:
\begin{Theorem}\label{Relative version of Theorem 0.2 of Lan18}
Let $X_{S}/\text{Spec }S$ be a smooth projective lifting of $X/\text{Spec }R$ and assume that $X_{S}$ admits a compatible system of smooth liftings $(X_{S})_{n}$ over $\text{Spec }W_{n}(S)$. Set $I_{S}:=VW_{n-1}(S)$. Let $\mathcal{F}il_{rel/R}^{r}\Omega_{(X_{S})_{n}/W_{n}(S)}^{\bullet}$ be the following complex:
\begin{center}
\resizebox{1.0\linewidth}{!}{
  \begin{minipage}{\linewidth}
\begin{align*}
I_{S}\mathcal{O}_{(X_{S})_{n}}\oplus
& \tilde{\mathfrak{a}}^{r}\mathcal{O}_{(X_{S})_{n}}\xrightarrow{pd\oplus d}I_{S}\Omega_{(X_{S})_{n}/W_{n}(S)}^{1}\oplus\tilde{\mathfrak{a}}^{r-1}\Omega_{(X_{S})_{n}/W_{n}(S)}^{1}\xrightarrow{pd\oplus d}\cdots \\
& \cdots\xrightarrow{pd\oplus d}I_{S}\Omega_{(X_{S})_{n}/W_{n}(S)}^{r-1}\oplus\tilde{\mathfrak{a}}\Omega_{(X_{S})_{n}/W_{n}(S)}^{r-1}\xrightarrow{d+d}\Omega_{(X_{S})_{n}/W_{n}(S)}^{r}\xrightarrow{d}\cdots
\end{align*}
\end{minipage}}
\end{center}
and let $\mathcal{N}^{r}_{rel/R}W_{n}\Omega_{X_{S}/S}^{\bullet}$ be the following complex:
\begin{center}
\resizebox{1.0\linewidth}{!}{
  \begin{minipage}{\linewidth}
\begin{align*}
(W_{n-1}\mathcal{O}_{X_{S}})_{[F]}\oplus
& 
\tilde{\mathfrak{a}}^{r}W_{n}\mathcal{O}_{X_{S}}\xrightarrow{d\oplus d}(W_{n-1}\Omega_{X_{S}/S}^{1})_{[F]}\oplus\tilde{\mathfrak{a}}^{r-1}W_{n}\Omega_{X_{S}/S}^{1}\xrightarrow{d\oplus d}\cdots \\
& \cdots\xrightarrow{d\oplus d}(W_{n-1}\Omega_{X_{S}/S}^{r-1})_{[F]}\oplus\tilde{\mathfrak{a}}W_{n}\Omega_{X_{S}/S}^{r-1}\xrightarrow{dV+d}W_{n}\Omega_{X_{S}/S}^{r}\xrightarrow{d}\cdots
\end{align*}
\end{minipage}}
\end{center}
Then for $r<p$ the complexes $\mathcal{F}il_{rel/R}^{r}\Omega_{(X_{S})_{n}/W_{n}(S)}^{\bullet}$ and $\mathcal{N}^{r}_{rel/R}W_{n}\Omega_{X_{S}/S}^{\bullet}$ are quasi-isomorphic.
\end{Theorem}
\begin{proof}
First notice that we may write $\mathcal{F}il_{rel/R}^{r}\Omega_{(X_{S})_{n}/W_{n}(S)}^{\bullet}$ and $\mathcal{N}^{r}_{rel/R}W_{n}\Omega_{X_{S}/S}^{\bullet}$ as the mapping cones of certain morphisms of complexes:
\begin{equation*}
\mathcal{F}il_{rel/R}^{r}\Omega_{(X_{S})_{n}/W_{n}(S)}^{\bullet}=\text{Cone}(\tilde{\mathfrak{a}}^{(r)}\Omega_{(X_{S})_{n}/W_{n}(S)}^{<r}[-1]\xrightarrow{f}\mathcal{F}^{r}\Omega_{(X_{S})_{n}/W_{n}(S)}^{\bullet})
\end{equation*}
and
\begin{equation*}
 \mathcal{N}^{r}_{rel/R}W_{n}\Omega_{X_{S}/S}^{\bullet}=\text{Cone}(\tilde{\mathfrak{a}}^{(r)}W_{n}\Omega_{X_{S}/S}^{<r}[-1]\xrightarrow{g}\mathcal{N}^{r}W_{n}\Omega_{X_{S}/S}^{\bullet})
\end{equation*}
where $\tilde{\mathfrak{a}}^{(r)}\Omega_{(X_{S})_{n}/W_{n}(S)}^{<r}$ and $\tilde{\mathfrak{a}}^{(r)}W_{n}\Omega_{X_{S}/S}^{<r}$ are the truncated complexes of the complexes $\tilde{\mathfrak{a}}^{(r)}\Omega_{(X_{S})_{n}/W_{n}(S)}^{\bullet}$
\begin{equation*}
\resizebox{1.0\hsize}{!}{$
\tilde{\mathfrak{a}}^{r}\mathcal{O}_{(X_{S})_{n}}\xrightarrow{-d}\tilde{\mathfrak{a}}^{r-1}\Omega_{(X_{S})_{n}/W_{n}(S)}^{1}\xrightarrow{-d}\cdots\xrightarrow{-d}\tilde{\mathfrak{a}}\Omega_{(X_{S})_{n}/W_{n}(S)}^{r-1}\xrightarrow{-d}\tilde{\mathfrak{a}}\Omega_{(X_{S})_{n}/W_{n}(S)}^{r}\xrightarrow{-d}\cdots
$}
\end{equation*}
and $\tilde{\mathfrak{a}}^{(r)}W_{n}\Omega_{X_{S}/S}^{\bullet}$
\begin{equation*}
\resizebox{1.0\hsize}{!}{$
\tilde{\mathfrak{a}}^{r}W_{n}\mathcal{O}_{X_{S}}\xrightarrow{-d}\tilde{\mathfrak{a}}^{r-1}W_{n}\Omega_{X_{S}/S}^{1}\xrightarrow{-d}\cdots\xrightarrow{-d}\tilde{\mathfrak{a}}W_{n}\Omega_{X_{S}/S}^{r-1}\xrightarrow{-d}\tilde{\mathfrak{a}}W_{n}\Omega_{X_{S}/S}^{r}\xrightarrow{-d}\cdots
$}
\end{equation*}
respectively. The morphisms $f$ and $g$ are respectively given by
\begin{equation*}
\tag{2.3}
\label{equation_2_3}
\begin{tikzpicture}[descr/.style={fill=white,inner sep=1.5pt}]
        \matrix (m) [
            matrix of math nodes,
            row sep=2.5em,
            column sep=2em,
            text height=1.5ex, text depth=0.25ex
        ]
        {  0 & \tilde{\mathfrak{a}}^{r}\mathcal{O} & \cdots & \tilde{\mathfrak{a}}^{2}\Omega^{r-2} &  \tilde{\mathfrak{a}}\Omega^{r-1} & 0 & \cdots \\
            I_{S}\mathcal{O} & I_{S}\Omega^{1} & \cdots & I_{S}\Omega^{r-1} & \Omega^{r} & \Omega^{r+1} & \cdots\\
        };

        \path[overlay,->, font=\scriptsize]
        (m-1-1) edge (m-2-1)
        (m-1-2) edge node [left]{$0$} (m-2-2)
        (m-1-4) edge node [left]{$0$} (m-2-4)
        (m-1-5) edge node [left]{$d$} (m-2-5)
        (m-1-6) edge (m-2-6);
        
        \path[overlay,->, font=\scriptsize]
        (m-1-1) edge (m-1-2)
        (m-1-2) edge node [above]{$-d$} (m-1-3)
        (m-1-3) edge node [above]{$-d$} (m-1-4)
        (m-1-4) edge node [above]{$-d$} (m-1-5)
        (m-1-5) edge (m-1-6)
        (m-1-6) edge (m-1-7)
        (m-2-1) edge node [above]{$pd$} (m-2-2)
        (m-2-2) edge node [above]{$pd$} (m-2-3)
        (m-2-3) edge node [above]{$pd$} (m-2-4)
        (m-2-4) edge node [above]{$d$} (m-2-5)
        (m-2-5) edge node [above]{$d$} (m-2-6)
        (m-2-6) edge node [above]{$d$} (m-2-7)
       ;
\end{tikzpicture}
\end{equation*}
and
\begin{equation*}
\begin{adjustbox}{width=12cm}
\tag{2.4}
\label{equation_2_4}
\begin{tikzpicture}[descr/.style={fill=white,inner sep=1.5pt}]
        \matrix (m) [
            matrix of math nodes,
            row sep=2.5em,
            column sep=0.5em,
            text height=1.5ex, text depth=0.25ex
        ]
        {  0 & \tilde{\mathfrak{a}}^{r}W_{n}\mathcal{O} & \cdots & \tilde{\mathfrak{a}}^{2}W_{n}\Omega^{r-2} &  \tilde{\mathfrak{a}}W_{n}\Omega^{r-1} & 0 & \cdots \\
            (W_{n-1}\mathcal{O})_{[F]} & (W_{n-1}\Omega^{1})_{[F]} & \cdots & (W_{n-1}\Omega^{r-1})_{[F]} & W_{n}\Omega^{r} & W_{n}\Omega^{r+1} & \cdots\\
        };

        \path[overlay,->, font=\scriptsize]
        (m-1-1) edge (m-2-1)
        (m-1-2) edge node [left]{$0$} (m-2-2)
        (m-1-4) edge node [left]{$0$} (m-2-4)
        (m-1-5) edge node [left]{$d$} (m-2-5)
        (m-1-6) edge (m-2-6);
        
        \path[overlay,->, font=\scriptsize]
        (m-1-1) edge (m-1-2)
        (m-1-2) edge node [above]{$-d$} (m-1-3)
        (m-1-3) edge node [above]{$-d$} (m-1-4)
        (m-1-4) edge node [above]{$-d$} (m-1-5)
        (m-1-5) edge (m-1-6)
        (m-1-6) edge (m-1-7)
        (m-2-1) edge node [above]{$d$} (m-2-2)
        (m-2-2) edge node [above]{$d$} (m-2-3)
        (m-2-3) edge node [above]{$d$} (m-2-4)
        (m-2-4) edge node [above]{$dV$} (m-2-5)
        (m-2-5) edge node [above]{$d$} (m-2-6)
        (m-2-6) edge node [above]{$d$} (m-2-7)
       ;
\end{tikzpicture}
\end{adjustbox}
\end{equation*}
(we briefly omitted the subscripts for typographical reasons). We will construct a morphism of distinguished triangles in the derived category
\begin{equation*}
\tag{2.5}
\label{equation_2_5}
\begin{tikzpicture}[descr/.style={fill=white,inner sep=1.5pt}]
        \matrix (m) [
            matrix of math nodes,
            row sep=2em,
            column sep=2em,
            text height=1.5ex, text depth=0.25ex
        ]
        {  \tilde{\mathfrak{a}}^{(r)}\Omega_{(X_{S})_{n}/W_{n}(S)}^{<r}[-1] & \mathcal{F}^{r}\Omega_{(X_{S})_{n}/W_{n}(S)}^{\bullet} & \mathcal{F}il_{rel/R}^{r}\Omega_{(X_{S})_{n}/W_{n}(S)}^{\bullet} & \ \\
            \tilde{\mathfrak{a}}^{(r)}W_{n}\Omega_{X_{S}/S}^{<r}[-1] & \mathcal{N}^{r}W_{n}\Omega_{X_{S}/S}^{\bullet} & \mathcal{N}_{rel/R}^{r}W_{n}\Omega_{X_{S}/S}^{\bullet} & \ \\
        };

        \path[overlay,->, font=\scriptsize]
        (m-1-1) edge (m-2-1)
        (m-1-2) edge (m-2-2)
        (m-1-3) edge (m-2-3);        
        
        \path[overlay,->, font=\scriptsize]
        (m-1-1) edge node [above] {$f$} (m-1-2)
        (m-1-2) edge (m-1-3)
        (m-1-3) edge node [above] {$+1$} (m-1-4)
        (m-2-1) edge node [above]{$g$} (m-2-2)
        (m-2-2) edge (m-2-3)
        (m-2-3) edge node [above] {$+1$} (m-2-4);

\end{tikzpicture}
\end{equation*}
where the middle vertical arrow is the quasi-isomorphism of Theorem \ref{Theorem 0.2 of Lan18} and the left vertical map is defined as follows:
\par
Assume first that there exists a closed embedding $(X_{S})_{n}\xrightarrow{i}Z_{n}$ into a projective smooth $W_{n}(S)$-scheme which is a Witt lift of $Z_{n}\times_{\text{Spec }W_{n}(S)}\text{Spec S}=Z_{S}$ in the sense of \cite{LZ04} Definition 3.3. Such a Witt lift always exists locally (\cite{LZ04} Prop. 3.2 and remarks after Def. 3.3) and induces maps $\mathcal{O}_{Z_{n}}\rightarrow W_{n}(\mathcal{O}_{Z_{S}})\rightarrow W_{n}(\mathcal{O}_{X_{S}})$. Let $\mathcal{O}_{D_{n}}$ be the PD-envelope of $i$ and let $\mathscr{I}^{[r]}$ be the divided power ideal with $\mathscr{I}=\ker(\mathcal{O}_{Z_{n}}\rightarrow\mathcal{O}_{(X_{S})_{n}})$. The comparison with crystalline cohomology yields a chain of quasi-isomorphisms \cite{BO78} Theorem 7.1 and \cite{LZ04} Theorem 3.5
\begin{equation*}
\tag{2.6}
\label{equation_2_6}
\Omega_{(X_{S})_{n}/W_{n}(S)}^{\bullet}\xleftarrow{\cong}\Omega_{D_{n}/W_{n}(S)}^{\bullet}\xrightarrow{\cong}W_{n}\Omega_{X_{S}/S}^{\bullet}
\end{equation*}
We construct a complex $\tilde{\mathfrak{a}}^{(r)}\Omega_{D_{n}/W_{n}(S)}^{<r}$ together with a diagram of maps
\begin{equation*}
\tag{2.7}
\label{equation_2_7}
\tilde{\mathfrak{a}}^{(r)}\Omega_{(X_{S})_{n}/W_{n}(S)}^{<r}\leftarrow\tilde{\mathfrak{a}}^{(r)}\Omega_{D_{n}/W_{n}(S)}^{<r}\rightarrow\tilde{\mathfrak{a}}^{(r)}W_{n}\Omega_{X_{S}/S}^{<r}
\end{equation*}
The argument is very similar to the proof of \cite{Lan18} Theorem 0.2. Consider the following diagram for $r<p$
\begin{equation*}
\begin{adjustbox}{width=12cm}
\tag{2.8}
\label{equation_2_8}
\begin{tikzpicture}[descr/.style={fill=white,inner sep=1.5pt}]
        \matrix (m) [
            matrix of math nodes,
            row sep=2em,
            column sep=1.5em,
            text height=1.5ex, text depth=0.25ex
        ]
        {  \tilde{\mathfrak{a}}^{r}\mathcal{O}_{D_{n}} & \ & \ & \ & \ & \ \\
           \tilde{\mathfrak{a}}^{r-1}\mathscr{I} & \tilde{\mathfrak{a}}^{r-1}\Omega_{D_{n}}^1 & \ & \ & \ & \ \\
           \vdots & \vdots & \ddots & \ & \ & \ \\
           \tilde{\mathfrak{a}}^{2}\mathscr{I}^{[r-2]}\mathcal{O}_{D_{n}} & \tilde{\mathfrak{a}}^{2}\mathscr{I}^{[r-3]}\Omega_{D_{n}}^{1} & \cdots & \tilde{\mathfrak{a}}^{2}\Omega_{D_{n}}^{r-2} & \ & \ \\
           \tilde{\mathfrak{a}}\mathscr{I}^{[r-1]}\mathcal{O}_{D_{n}} & \tilde{\mathfrak{a}}\mathscr{I}^{[r-2]}\Omega_{D_{n}}^1 & \cdots & \cdots & \tilde{\mathfrak{a}}\Omega_{D_{n}}^{r-1}  & \tilde{\mathfrak{a}}\Omega_{D_{n}}^{r} & \cdots \\
                  };

        \path[overlay,->, font=\scriptsize]
        (m-1-1) edge node [above right] {$-d$} (m-2-2)
        (m-2-2) edge node [above right] {$-d$} (m-3-3)
        (m-3-3) edge node [above right] {$-d$} (m-4-4)
        (m-4-4) edge node [above right] {$-d$} (m-5-5);

        \path[overlay,->,font=\scriptsize]
        (m-2-1) edge node [above] {$-d$} (m-2-2)
        (m-4-1) edge node [above] {$-d$} (m-4-2)
        (m-4-2) edge node [above] {$-d$} (m-4-3)
        (m-4-3) edge node [above] {$-d$} (m-4-4)
        (m-5-1) edge node [above] {$-d$} (m-5-2)
        (m-5-2) edge node [above] {$-d$} (m-5-3)
        (m-5-4) edge node [above] {$-d$} (m-5-5)
        (m-5-5) edge node [above] {$-d$} (m-5-6)
        (m-5-6) edge node [above] {$-d$} (m-5-7);       
       
\end{tikzpicture}
\end{adjustbox}
\end{equation*}
The complex $\mathscr{I}^{[s]}\rightarrow\mathscr{I}^{[s-1]}\Omega_{D_{n}}^{1}\rightarrow\cdots\rightarrow\Omega_{D_{n}}^{s}\rightarrow\Omega_{D_{n}}^{s+1}$ is exact in degrees $<s$ and quasi-isomorphic to $\Omega_{(X_{S})_{n}/W_{n}(S)}^{\geq s}[-s]$ by \cite{BO78} Theorem 7.2. Since all $\mathscr{I}^{[l]}\Omega_{D_{n}}^{k}$ are - locally - free $\mathcal{O}_{(X_{S})_{n}}$-modules by \cite{BO78} Prop. 3.32, the complexes $\mathscr{I}^{[s-\bullet]}\Omega_{D_{n}}^{\bullet}$ remain exact in degrees $<s$ after $\otimes_{W_{n}(S)}S$ and $\otimes_{W_{n}(S)}R$; they coincide then with the corresponding complexes for the embeddings $X_{S}\rightarrow Z_{S}$, resp $X\rightarrow Z_{S}\times_{\text{Spec }S}\text{Spec }R$, since $\tilde{\mathfrak{a}}$ is a direct summand of $\ker(W_{n}(S)\rightarrow R)$ the lower horizontal sequence in (\ref{equation_2_8}) is exact in degrees $<r$ and quasi-isomorphic to $\tilde{\mathfrak{a}}\Omega_{(X_{S})_{n}/W_{n}(S)}^{\geq r-1}[-(r-1)]$. By an easy induction argument (replace $R$ by $S/\mathfrak{a}^{l}$) one sees that the other horizontal sequences in (\ref{equation_2_8}) are - up to the diagonal term - exact as well. The degree wise sum of the two lower sequences is quasi-isomorphic to
\begin{equation*}
\left(\tilde{\mathfrak{a}}^{2}\Omega_{(X_{S})_{n}/W_{n}(S)}^{r-2}\xrightarrow{-d}\tilde{\mathfrak{a}}\Omega_{(X_{S})_{n}/W_{n}(S)}^{r-1}\xrightarrow{-d}\tilde{\mathfrak{a}}\Omega_{(X_{S})_{n}/W_{n}(S)}^{r}\xrightarrow{-d}\cdots\right)[-(r-2)]
\end{equation*}
Finally, the degree wise sum of all horizontal sequences yields a complex denoted by $\tilde{\mathfrak{a}}^{(r)}\Omega_{D_{n}/W_{n}(S)}^{\bullet}$ which is quasi-isomorphic to $\tilde{\mathfrak{a}}^{(r)}\Omega_{(X_{S})_{n}/W_{n}(S)}^{\bullet}$. From the above we conclude that the complex $\tilde{\mathfrak{a}}\mathcal{F}il^{r}\Omega_{D_{n}}^{\bullet}$
\begin{equation*}
\tilde{\mathfrak{a}}\mathcal{J}^{[r]}\mathcal{O}_{D_{n}}\xrightarrow{-d}\tilde{\mathfrak{a}}\mathcal{J}^{[r-1]}\Omega_{D_{n}}^{1}\xrightarrow{-d}\cdots\xrightarrow{-d}\tilde{\mathfrak{a}}\Omega_{D_{n}}^{r}\xrightarrow{-d}\tilde{\mathfrak{a}}\Omega_{D_{n}}^{r+1}\xrightarrow{-d}\cdots
\end{equation*}
is quasi-isomorphic to
\begin{equation*}
(\tilde{\mathfrak{a}}\Omega_{(X_{S})_{n}/W_{n}(S)}^{r}\xrightarrow{-d}\tilde{\mathfrak{a}}\Omega_{(X_{S})_{n}/W_{n}(S)}^{r+1}\xrightarrow{-d}\tilde{\mathfrak{a}}\Omega_{(X_{S})_{n}/W_{n}(S)}^{r+2}\xrightarrow{-d}\cdots)[-r]
\end{equation*}
The natural embedding of $\tilde{\mathfrak{a}}\mathcal{F}il^{r}\Omega_{D_{n}}^{\bullet}$ into the lower horizontal complex in diagram (\ref{equation_2_8}) defines an injective map
\begin{equation*}
\tilde{\mathfrak{a}}\mathcal{F}il^{r}\Omega_{D_{n}}^{\bullet}\rightarrow\tilde{\mathfrak{a}}^{(r)}\Omega_{D_{n}/W_{n}(S)}^{\bullet}
\end{equation*}
We denote the mapping cone of this map by $\tilde{\mathfrak{a}}^{(r)}\Omega_{D_{n}/W_{n}(S)}^{<r}$. The notation is justified because the complex vanishes in degrees $\geq r$. We see that $\tilde{\mathfrak{a}}^{(r)}\Omega_{D_{n}/W_{n}(S)}^{<r}$ is quasi-isomorphic to $\tilde{\mathfrak{a}}^{(r)}\Omega_{(X_{S})_{n}/W_{n}(S)}^{<r}$.
\par
Note that under the canonical map $\mathcal{O}_{D_{n}}\rightarrow W_{n}(\mathcal{O}_{X_{S}})$ the image of $\mathscr{I}$ is contained in $VW_{n-1}(\mathcal{O}_{X_{S}})$ and hence the image of $\mathfrak{a}\cdot\mathscr{I}$ is zero in $W_{n}(\mathcal{O}_{X_{S}})$. Hence the map $\mathcal{O}_{D_{n}}\rightarrow W_{n}(\mathcal{O}_{X_{S}})$, compatible with Frobenius, induces a well-defined map
\begin{equation*}
\tilde{\mathfrak{a}}^{(r)}\Omega_{D_{n}/W_{n}(S)}^{\bullet}\rightarrow\tilde{\mathfrak{a}}^{(r)}W_{n}\Omega_{X_{S}/S}^{\bullet}
\end{equation*}
If one has two embeddings $(X_{S})_{n}\xrightarrow{i}Z_{n}$, $(X_{S})_{n}\xrightarrow{i'}Z'_{n}$ into Witt lifts then by considering the product embedding $(X_{S})_{n}\xrightarrow{(i,i')}Z_{n}\times Z'_{n}$ we get a well-defined map in the derived category
\begin{equation*}
\tilde{\mathfrak{a}}^{(r)}\Omega_{(X_{S})_{n}/W_{n}(S)}^{\bullet}\rightarrow\tilde{\mathfrak{a}}^{(r)}W_{n}\Omega_{X_{S}/S}^{\bullet}
\end{equation*}
which induces a map 
\begin{equation*}
\tilde{\mathfrak{a}}^{(r)}\Omega_{(X_{S})_{n}/W_{n}(S)}^{<r}\rightarrow\tilde{\mathfrak{a}}^{(r)}W_{n}\Omega_{X_{S}/S}^{<r}
\end{equation*}
on the truncated complexes. This defines (\ref{equation_2_7}).
\par
To show that it is a quasi-isomorphism is a Zariski-local question on $X_{S}$, so it suffices to check the case where $X_{S}=\text{Spec B}$ is affine and $B$ is \'{e}tale over a polynomial algebra $A:=S[T_{1},\ldots,T_{d}]$. Set $A_{n}:=W_{n}(S)[T_{1},\ldots,T_{d}]$ and let $\phi_{n}:A_{n}\rightarrow A_{n-1}$ be the map extending $F:W_{n}(S)\rightarrow W_{n-1}(S)$ given by setting $\phi_{n}(T_{i})=T_{i}^{p}$, and let $\delta_{n}:A_{n}\rightarrow W_{n}(A)$ be the unique $W_{n}(S)$-algebra homomorphism which sends each $T_{i}$ to its Teichm\"{u}ller representative. Then the data $(A_{n},\phi_{n},\delta_{n})$ is a Frobenius lift of $A$ to $W(S)$ (see \S3 of \cite{LZ04}). Since $A\rightarrow B$ is \'{e}tale, there exists a unique set of liftings $B_{n}$ of $B$ which are each \'{e}tale over $A_{n}$, and homomorphisms $\psi_{n}:B_{n}\rightarrow B_{n-1}$, $\epsilon_{n}:B_{n}\rightarrow W_{n}(B)$ which are compatible with $\phi_{n},\delta_{n}$. The morphism $\tilde{\mathfrak{a}}^{(r)}\Omega_{B_{n}/W_{n}(S)}^{\bullet}\rightarrow\tilde{\mathfrak{a}}^{(r)}W_{n}\Omega_{B/S}^{\bullet}$ is the one induced by the $\epsilon_{n}$.
\par
First let us treat the special case that $B=A=S[T_{1},\ldots,T_{d}]$; the quasi-isomorphism is easy to establish because in this case the de Rham-Witt complex has a rather explicit description (originally due to Illusie in the case of a perfect field, and by \S2 of \cite{LZ04} in our generality). Indeed, the de Rham-Witt complex decomposes into a direct sum of an integral part and an acyclic fractional part (\cite{LZ04}, (3.9)) and the fractional part is contained in the image of $V$ resp. $dV$. We must check that the fractional part is still acyclic after multiplying by the logarithmic Teichm\"{u}ller ideal $\tilde{\mathfrak{a}}$. But multiplying anything in the image of $V$ by an element $a$ means applying Frobenius to $a$, and Frobenius kills $\tilde{\mathfrak{a}}$ by \cite{Zin02} Lemma 38, so we conclude that the fractional part of the de Rham-Witt complex is annihilated by $\tilde{\mathfrak{a}}$. 
\par
Now we return to the general case where $X_{S}=\text{Spec }B$ for $B$ \'{e}tale over $A=S[T_{1},\ldots,T_{d}]$. Choose an integer $m$ such that $p^{m}W_{n}(S)=0$ and set $\phi^{m}:=\phi_{m+n}\circ\cdots\circ\phi_{n+1}:A_{m+n}\rightarrow A_{n}$. Then for each $l$ we have an isomorphism (see the proof of \cite{LZ04} Theorem 3.5)
\begin{equation*}
\Omega_{B_{n}/W_{n}(S)}^{l}\cong B_{n}\otimes_{A_{n}}\Omega_{A_{n}/W_{n}(S)}^{l}\cong B_{m+n}\otimes_{A_{m+n},\phi^{m}}\Omega_{A_{n}/W_{n}(S)}^{l}
\end{equation*}
and, likewise, for each $j$, $l$ an isomorphism
\begin{equation*}
\tilde{\mathfrak{a}}^{j}\Omega_{B_{n}/W_{n}(S)}^{l}\cong B_{n}\otimes_{A_{n}}\tilde{\mathfrak{a}}^{j}\Omega_{A_{n}/W_{n}(S)}^{l}\cong B_{m+n}\otimes_{A_{m+n},\phi^{m}}\tilde{\mathfrak{a}}^{j}\Omega_{A_{n}/W_{n}(S)}^{l}
\end{equation*}
and this gives an isomorphism of complexes
\begin{equation*}
\tilde{\mathfrak{a}}^{(r)}\Omega_{B_{n}/W_{n}(S)}^{\bullet}\cong B_{m+n}\otimes_{A_{m+n},\phi^{m}}\tilde{\mathfrak{a}}^{(r)}\Omega_{A_{n}/W_{n}(S)}^{\bullet}
\end{equation*}
if we give the right-hand side the differential $1\otimes -d$. 
\par
Let 
\begin{equation*}
W_{n}\Omega_{A/S}^{\bullet}=W_{n}\Omega_{A/S}^{\integral,\bullet}\oplus W_{n}\Omega_{A/S}^{\fractional,\bullet}
\end{equation*}
be the decomposition into integral and fractional parts, as mentioned above. Note that this is a direct sum decomposition of complexes of $A_{m+n}$-modules via restriction of scalars $A_{m+n}\xrightarrow{\phi^{m}}A_{n}$. Then using the following facts proven in \cite{LZ04} Theorem 3.5 and Prop. 1.7
\begin{enumerate}
\item base change of $W_{n}\Omega^{\bullet}$ for \'{e}tale maps
\item $W_{n}\Omega_{B/S}^{\bullet}\cong B_{m+n}\otimes_{A_{m+n},\phi^{m}}W_{n}\Omega_{A/S}^{\bullet}$
\end{enumerate}
we have
\begin{align*}
W_{n}\Omega_{B/S}^{\bullet}
& = (B_{m+n}\otimes_{A_{m+n}}W_{n}\Omega_{A/S}^{\integral,\bullet})\oplus (B_{m+n}\otimes_{A_{m+n}}W_{n}\Omega_{A/S}^{\fractional,\bullet}) \\
& =:W_{n}\Omega_{B/S}^{\integral,\bullet}\oplus W_{n}\Omega_{B/S}^{\fractional,\bullet}
\end{align*}
Since $W_{n}\Omega_{A/S}^{\fractional,\bullet}$ is acyclic and $B_{m+n}$ is a flat $A_{m+n}$-module, $W_{n}\Omega_{B/S}^{\fractional,\bullet}$ is acyclic too.
\par
Then we define complexes $\tilde{\mathfrak{a}}^{(r)}W_{n}\Omega_{B/S}^{\bullet}$, $\tilde{\mathfrak{a}}^{(r)}W_{n}\Omega_{B/S}^{\integral,\bullet}$ and $\tilde{\mathfrak{a}}^{(r)}W_{n}\Omega_{B/S}^{\fractional,\bullet}$ in exactly the same manner as at the beginning of the proof. Evidently we get a direct sum decomposition 
\begin{equation*}
\tilde{\mathfrak{a}}^{(r)}W_{n}\Omega_{B/S}^{\bullet}=\tilde{\mathfrak{a}}^{(r)}W_{n}\Omega_{B/S}^{\integral,\bullet}\oplus \tilde{\mathfrak{a}}^{(r)}W_{n}\Omega_{B/S}^{\fractional,\bullet}
\end{equation*}
Since $W_{n}\Omega_{A/S}^{\integral,\bullet}\cong\Omega_{A_{n}/W_{n}(S)}^{\bullet}$, we get an isomorphism
\begin{align*}
\tilde{\mathfrak{a}}^{(r)}W_{n}\Omega_{B/S}^{\integral,\bullet}
& \cong B_{m+n}\otimes_{A_{m+n},\phi^{m}}\tilde{\mathfrak{a}}^{(r)}W_{n}\Omega_{A/S}^{\integral,\bullet} \\
& \cong B_{m+n}\otimes_{A_{m+n},\phi^{m}}\tilde{\mathfrak{a}}^{(r)}\Omega_{A_{n}/W_{n}(S)}^{\bullet} \\
& \cong \tilde{\mathfrak{a}}^{(r)}\Omega_{B_{n}/W_{n}(S)}^{\bullet}
\end{align*}
Since $\tilde{\mathfrak{a}}$ annihilates the fractional part $W_{n}\Omega_{A/S}^{\bullet}$ as observed above, we get that $\tilde{\mathfrak{a}}^{(r)}W_{n}\Omega_{B/S}^{\fractional,\bullet}$ vanishes. Hence we obtain an isomorphism
\begin{equation*}
\tilde{\mathfrak{a}}^{(r)}W_{n}\Omega_{B/S}^{\bullet}\cong\tilde{\mathfrak{a}}^{(r)}\Omega_{B_{n}/W_{n}(S)}^{\bullet}
\end{equation*}
and likewise for the truncated complexes
\begin{equation*}
\tilde{\mathfrak{a}}^{(r)}W_{n}\Omega_{B/S}^{<r}\cong\tilde{\mathfrak{a}}^{(r)}\Omega_{B_{n}/W_{n}(S)}^{<r}
\end{equation*}
 as desired. Since the construction of this isomorphism using PD-envelopes of embeddings into Witt lifts is compatible with the construction of the comparison map 
\begin{equation*}
\mathcal{F}^{r}\Omega_{(X_{S})_{n}/W_{n}(S)}^{\bullet}\rightarrow\mathcal{N}^{r}W_{n}\Omega_{X/S}^{\bullet}
\end{equation*}
(\cite{Lan18}, Theorem 0.2) the diagram (\ref{equation_2_5}) commutes on the left. 

Let $\mathcal{F}il^{r}\Omega_{D_{n}/W_{n}(S)}^{\bullet}$ be the complex constructed in the proof of \cite[Theorem 0.2]{Lan18} (see page 1868). Analogously to the map 
\begin{equation*}
f:\tilde{\mathfrak{a}}^{(r)}\Omega_{(X_{S})_{n}/W_{n}(S)}^{<r}[-1]\rightarrow\mathcal{F}^{r}\Omega_{(X_{S})_{n}/W_{n}(S)}^{\bullet}
\end{equation*}
in \eqref{equation_2_3}, one can define a canonical map
\begin{equation*}
\tilde{f}:\tilde{\mathfrak{a}}^{(r)}\Omega_{D_{n}/W_{n}(S)}^{<r}[-1]\rightarrow\mathcal{F}il^{r}\Omega_{D_{n}/W_{n}(S)}^{\bullet}
\end{equation*}
which is the zero map in degrees $\neq r$ and equal to $d$ in degree $r$. Let $\mathcal{F}il^{r}_{rel/R}\Omega_{D_{n}/W_{n}(S)}^{\bullet}$ be the mapping cone of $\tilde{f}$. One obtains a commutative diagram of complexes
\begin{equation*}
\begin{adjustbox}{width=12cm}
\tag{2.9}
\label{equation_2_9}
\begin{tikzpicture}[descr/.style={fill=white,inner sep=1.5pt}]
        \matrix (m) [
            matrix of math nodes,
            row sep=2.5em,
            column sep=2.5em,
            text height=1.5ex, text depth=0.25ex
        ]
        { \tilde{\mathfrak{a}}^{(r)}\Omega_{(X_{S})_{n}/W_{n}(S)}^{<r}[-1] & \mathcal{F}^{r}\Omega_{(X_{S})_{n}/W_{n}(S)}^{\bullet} & \mathcal{F}il^{r}_{rel/R}\Omega_{(X_{S})_{n}/W_{n}(S)}^{\bullet} & \cdots \\
        \tilde{\mathfrak{a}}^{(r)}\Omega_{D_{n}/W_{n}(S)}^{<r}[-1] & \mathcal{F}il^{r}\Omega_{D_{n}/W_{n}(S)}^{\bullet} & \mathcal{F}il^{r}_{rel/R}\Omega_{D_{n}/W_{n}(S)}^{\bullet} & \cdots \\
        \tilde{\mathfrak{a}}^{(r)}W_{n}\Omega_{X_{S}/S}^{<r}[-1] & \mathcal{N}^{r}W_{n}\Omega_{X_{S}/S}^{\bullet} & \mathcal{N}^{r}_{rel/R}W_{n}\Omega_{X_{S}/S}^{\bullet} & \cdots \\
                 };

        \path[overlay,->, font=\scriptsize]
        (m-1-1) edge node [above] {$f$} (m-1-2)
        (m-1-2) edge (m-1-3)
        (m-1-3) edge node [above] {$+1$} (m-1-4)
        (m-2-1) edge node [above] {$\tilde{f}$} (m-2-2)
        (m-2-2) edge (m-2-3)
        (m-2-3) edge node [above] {$+1$} (m-2-4)
        (m-3-1) edge node [above] {$g$} (m-3-2)
        (m-3-2) edge (m-3-3)
        (m-3-3) edge node [above] {$+1$} (m-3-4)
        (m-2-1) edge (m-1-1)
        (m-2-1) edge (m-3-1)
        (m-2-2) edge (m-1-2)
        (m-2-2) edge (m-3-2)
        (m-2-3) edge (m-1-3)
        (m-2-3) edge (m-3-3);
        
\end{tikzpicture} 
\end{adjustbox}
\end{equation*}
where the right vertical arrows are canonical maps induced on the level of mapping cones by the commutative diagrams of complexes on the left. Since the vertical arrows on the left and in the middle of the diagram are quasi-isomorphisms by construction, the vertical arrows on the right hand side are also quasi-isomorphisms. This proves that \eqref{equation_2_5} is a morphism of distinguished triangles in the derived category where all vertical arrows are isomorphisms.

In the absence of a global embedding into a Witt lift one proceeds by simplicial methods as in the proof of \cite{Lan18} Theorem 0.2, \cite{LZ04} \S3.2 and \cite{Ill79} II.1. to obtain Theorem \ref{Relative version of Theorem 0.2 of Lan18}. For the convenience of the reader we recall the argument. Let $(X_{S})_{n}(i)$, $i\in I$ be a covering of $(X_{S})_{n}$ inducing a covering $X_{S}(i)$ of $X_{S}$ and an embedding $(X_{S})_{n}(i)\rightarrow(Z_{S})_{n}(i)$ which is a Witt lift of $Z_{S}(i):=(Z_{S})_{n}(i)\times_{W_{n}(S)}S$. We set
\begin{equation*}
(X_{S})_{n}(i_{1},\ldots,i_{r})=(X_{S})_{n}(i_{1})\cap\cdots\cap(X_{S})_{n}(i_{r})
\end{equation*} 
(and likewise for $X_{S}$ itself) and
\begin{equation*}
(Z_{S})_{n}(i_{1},\ldots,i_{r})=(Z_{S})_{n}(i_{1})\times_{W_{n}(S)}\cdots\times_{W_{n}(S)}(Z_{S})_{n}(i_{r})
\end{equation*} 
We denote by $D_{n}(i_{1},\ldots,i_{r})$ the PD-envelope of the canonical morphism $(X_{S})_{n}(i_{1},\ldots,i_{r})\rightarrow (Z_{S})_{n}(i_{1},\ldots,i_{r})$. One gets simplicial schemes $X_{S}^{\bullet}\rightarrow(X_{S})_{n}^{\bullet}\rightarrow D_{n}^{\bullet}\rightarrow(Z_{S})_{n}^{\bullet}$ and an isomorphism in the derived category of simplicial complexes of sheaves on $X_{S}^{\bullet}$
\begin{equation*}
\mathcal{F}il_{rel/R}^{r}\Omega_{(X_{S})_{n}^{\bullet}}^{\bullet}\rightarrow\mathcal{N}_{rel/R}^{r}W_{n}\Omega_{X_{S}^{\bullet}/S}^{\bullet}
\end{equation*}
Let $X_{S}^{\bullet}\xrightarrow{\theta}X_{S}$ be the natural augmentation. By applying $R\theta_{\ast}$ we get by cohomological descent in the Zariski topology the desired isomorphism in Theorem \ref{Relative version of Theorem 0.2 of Lan18}. 
\end{proof}
We now prove Theorem \ref{Relative displays theorem}(a). As in Theorem \ref{Theorem 0.2 of Lan18}, the isomorphisms between the complexes in Theorem \ref{Relative version of Theorem 0.2 of Lan18} are compatible for varying $n$. One first assumes the existence of a compatible system of embeddings into Witt lifts; in the general case one uses again simplicial methods as outlined in \cite{Lan18} to obtain an isomorphism of procomplexes $\mathcal{F}il_{rel/R}^{r}\Omega_{(X_{S})_{\bullet}/W_{\bullet}(S)}^{\bullet}\cong\mathcal{N}_{rel/R}^{r}W_{\bullet}\Omega_{X_{S}/S}^{\bullet}$. Let
\begin{equation*}
(\mathcal{P}_{S/R})_{r}=\mathbb{H}^{n}(X_{S},\mathcal{N}^{r}_{rel/R}W_{\bullet}\Omega_{X_{S}/S}^{\bullet})=\mathbb{H}^{n}(X_{S},\mathcal{F}il_{rel/R}^{r}\Omega_{(X_{S})_{\bullet}/W_{\bullet}(S)}^{\bullet})
\end{equation*}
Using the same argument as in the proof of Theorem \ref{Display theorem}(a), we see that the $E_{1}$-spectral sequence associated to the complex $\mathcal{F}il_{rel/R}^{r}\Omega_{(X_{S})_{\bullet}/W_{\bullet}(S)}^{\bullet}$ degenerates. This implies a decomposition
\begin{equation*}
(\mathcal{P}_{S/R})_{r}=\mathcal{J}_{r}L_{0}\oplus\mathcal{J}_{r-1}L_{1}\oplus\cdots\oplus\mathcal{J}L_{r-1}\oplus L_{r}\oplus\cdots\oplus L_{n} 
\end{equation*}
where $L_{i}=H^{n-i}((X_{S})_{\bullet},\Omega_{(X_{S})_{\bullet}/W_{\bullet}(S)}^{i})$ and $\mathcal{J}_{i}=\tilde{\mathfrak{a}}^{i}\oplus I_{S}$ (compare the construction of standard displays over the relative Witt frame $\mathcal{W}_{S/R}$ in \cite{LZ19} and Appendix, Def. \ref{standard}).
\par
The maps $\hat{F}_{r}:(\mathcal{P}_{S/R})_{r}\rightarrow (\mathcal{P}_{S/R})_{0}=H_{\cris}^{n}(X/W(S))$ induce maps $\Phi_{r}:L_{r}\rightarrow (\mathcal{P}_{S})_{0}$ by $\Phi_{r}=\hat{F}_{r}|L_{r}$. To show that $( (\mathcal{P}_{S/R})_{r},\hat{F}_{r},\hat{\iota}_{r},\hat{\alpha}_{r})$ defines a relative display on $H_{\cris}^{n}(X/W(S))$ is equivalent to the condition that 
\begin{equation*}
\bigoplus_{i=0}^{n}\Phi_{i}:L_{0}\oplus\cdots\oplus L_{n}\rightarrow (\mathcal{P}_{S})_{0}
\end{equation*}
is a $\sigma$-linear isomorphism. The argument is the same as in the proof of Theorem \ref{Display theorem}(a). On the level of standard displays (i.e. displays given by standard data), the relative display associated to the display $\mathcal{P}_{S}$ is given by the inclusions (Appendix, Remark \ref{base change})
\begin{equation*}
I_{S}L_{0}\oplus\cdots\oplus I_{S}L_{r-1}\oplus L_{r}\oplus\cdots\oplus L_{n}\hookrightarrow\mathcal{J}_{r}L_{0}\oplus\cdots\oplus\mathcal{J}L_{r-1}\oplus L_{r}\oplus\cdots\oplus L_{n}
\end{equation*}
Since the chain of quasi-isomorphisms between $\mathcal{F}^{r}\Omega_{(X_{S})_{n}/W_{n}(S)}^{\bullet}$ and $\mathcal{N}^{r}W_{n}\Omega_{X_{S}/S}^{\bullet}$ and between $\mathcal{F}il_{rel/R}^{r}\Omega_{(X_{S})_{n}/W_{n}(S)}^{\bullet}$ and $\mathcal{N}^{r}_{rel/R}W_{n}\Omega_{X_{S}/S}^{\bullet}$ are compatible under the inclusion maps (see the commutative diagram (\ref{equation_2_5}))
\begin{equation*}
\mathcal{F}^{r}\Omega_{(X_{S})_{n}/W_{n}(S)}^{\bullet}\hookrightarrow\mathcal{F}il_{rel/R}^{r}\Omega_{(X_{S})_{n}/W_{n}(S)}^{\bullet} 
\end{equation*}
and
\begin{equation*}
\mathcal{N}^{r}W_{n}\Omega_{X_{S}/S}^{\bullet}\hookrightarrow\mathcal{N}^{r}_{rel/R}W_{n}\Omega_{X_{S}/S}^{\bullet}
\end{equation*}
we conclude that $u_{\ast}\mathcal{P}_{S}=\mathcal{P}_{S/R}$. This proves Theorem \ref{Relative displays theorem}(a).
\subsection{Proof of Theorem \ref{Display theorem}(b)}
Let $A\rightarrow R$ be a frame for $R$ such that the kernel $\mathfrak{a}$ is equipped with divided powers, by definition $A$ is equipped with a lifting $\sigma:A\rightarrow A$ of the Frobenius $A/pA\rightarrow A/pA$. We consider the Cartier map $A\rightarrow W(A)$ into the Witt ring (\cite{Ill79}, 0.1.3.16). Then  $A\rightarrow R$ factors through
\begin{equation*}
A\rightarrow W(A)\rightarrow W(R)\rightarrow R
\end{equation*}
The kernel $\mathcal{J}$ of $W(A)\rightarrow R$ is $\tilde{\mathfrak{a}}\oplus VW(A)=\tilde{\mathfrak{a}}\oplus I_{A}$, where $\tilde{\mathfrak{a}}$ is the logarithmic Teichm\"{u}ller ideal, equipped again with divided powers. We then get a second frame $(W(A),\mathcal{J},\sigma,\dot{\sigma})$ for $R$, where $\sigma$ is the Frobenius on $W(A)$ and $\dot{\sigma}:\mathcal{J}\rightarrow W(A)$, $a+V\xi\mapsto\xi$. This is the definition of the relative Witt frame $\mathcal{W}_{A/R}$ (Appendix, Def. \ref{relative Witt frame}).
\par
Assuming the existence of liftings $\mathcal{Y}/\text{Spf }A$ of $X$ that satisfy (A1) and (A2), we get by base change liftings $\tilde{\mathcal{Y}}/\text{Spf }W(A)$ that also satisfy (A1) and (A2). It is therefore enough to show Theorem \ref{Display theorem}(b) by working with the relative Witt frame $\mathcal{W}_{A/R}$ and the lifting $\tilde{\mathcal{Y}}$. Then a window over $W(A)$ (\cite{LZ07}, Def. 5.1 and Appendix, Def. \ref{window}) is the same as a display $\mathcal{P}_{A/R}$ over the relative Witt frame $\mathcal{W}_{A/R}$. We denote now by $\mathcal{P}_{A/R}$ the display associated to the lifting $\tilde{\mathcal{Y}}$ that exists by (\cite{LZ07}, Thm. 5.5). Let $Y_{m,s}:=\tilde{\mathcal{Y}}\times_{W(A)}W_{s}(A/p^{m})$. Then
\begin{equation*}
(\mathcal{P}_{A/R})_{r}=\varprojlim_{s,m}H_{\cris}^{n}(X,\mathcal{J}_{X/W_{s}(A/p^{m})}^{[r]})
\end{equation*}
is equipped with a divided Frobenius $F_{r}=\frac{F}{p^{r}}$ where $F$ is the Frobenius on crystalline cohomology. Assume $A\rightarrow R$ factors through $A/p^{m}\rightarrow R$. By (\cite{BO78}, Thm. 7.2) the groups
\begin{equation*}
H_{\cris}^{n}(X,\mathcal{J}_{X/W(A/p^{m})}^{[r]})=\varprojlim_{s}H_{\cris}^{n}(X,\mathcal{J}_{X/W_{s}(A/p^{m})}^{[r]})
\end{equation*}
are the hypercohomology groups of the procomplexes $\mathcal{F}il^{[r]}\Omega_{Y_{\bullet,m}/W_{\bullet}(A/p^{m})}^{\bullet}$ defined as follows:
\begin{center}
\resizebox{1.0\linewidth}{!}{
  \begin{minipage}{\linewidth}
\begin{align*}
(\tilde{\mathfrak{a}}_{m}^{[r]}\oplus p^{r-1}I_{A/p^{m}})
& \Omega_{Y_{\bullet,m}/W_{\bullet}(A/p^{m})}^{0}\xrightarrow{d}(\tilde{\mathfrak{a}}_{m}^{[r-1]}\oplus p^{r-2}I_{A/p^{m}})\Omega_{Y_{\bullet,m}/W_{\bullet}(A/p^{m})}^{1}\xrightarrow{d}\cdots \\
& \cdots\xrightarrow{d}(\tilde{\mathfrak{a}}_{m}\oplus I_{A/p^{m}})\Omega_{Y_{\bullet,m}/W_{\bullet}(A/p^{m})}^{r-1}\xrightarrow{d}\Omega_{Y_{\bullet,m}/W_{\bullet}(A/p^{m})}^{r}\xrightarrow{d}\cdots
\end{align*}
\end{minipage}}
\end{center}
where $\tilde{\mathfrak{a}}_{m}$ is the logarithmic Teichm\"{u}ller ideal associated to $\mathfrak{a}_{m}:=\ker(A/p^{m}\rightarrow R)$ and $I_{A/p^{m}}:=VW(A/p^{m})$, and we have used that $\tilde{\mathfrak{a}}_{m}\cdot I_{A/p^{m}}=0$ and for the ideal $\mathcal{J}_{m}=\ker(W(A/p^{m})\rightarrow R)$ we have $\mathcal{J}_{m}^{[s]}=\tilde{\mathfrak{a}}_{m}^{[s]}\oplus p^{s-1}I_{A/p^{m}}$.
\par 
As $A$ and $W(A)$ are $p$-torsion free, multiplication by $p$ on $W(A)$ and the pro-group $W_{\bullet}(A/p^{\bullet})$ is injective, hence the procomplexes $\mathcal{F}il^{[r]}\Omega_{Y_{\bullet,\bullet}/W_{\bullet}(A/p^{\bullet})}^{\bullet}$ and $\mathcal{F}il_{rel/R}^{[r]}\Omega_{Y_{\bullet,\bullet}/W_{\bullet}(A/p^{\bullet})}^{\bullet}$, defined as 
\begin{align*}
(\tilde{\mathfrak{a}}_{\bullet}^{[r]}\oplus I_{A/p^{\bullet}})
& \Omega_{Y_{\bullet,\bullet}/W_{\bullet}(A/p^{\bullet})}^{0}\xrightarrow{d\oplus pd}(\tilde{\mathfrak{a}}_{\bullet}^{[r-1]}\oplus I_{A/p^{\bullet}})\Omega_{Y_{\bullet,\bullet}/W_{\bullet}(A/p^{\bullet})}^{1}\xrightarrow{d\oplus pd}\cdots \\
& \cdots\xrightarrow{d\oplus pd}(\tilde{\mathfrak{a}}_{\bullet}\oplus I_{A/p^{\bullet}})\Omega_{Y_{\bullet,\bullet}/W_{\bullet}(A/p^{\bullet})}^{r-1}\xrightarrow{d}\Omega_{Y_{\bullet,\bullet}/W_{\bullet}(A/p^{\bullet})}^{r}\xrightarrow{d}\cdots
\end{align*}
are isomorphic. By Theorem \ref{Relative version of Theorem 0.2 of Lan18}, the procomplexes $\mathcal{F}il_{rel/R}^{[r]}\Omega_{Y_{\bullet,\bullet}/W_{\bullet}(A/p^{\bullet})}^{\bullet}$ and $\mathcal{N}_{rel/R}^{r}W_{\bullet}\Omega_{Y_{\bullet}/(A/p^{\bullet})}^{\bullet}$, defined as
\begin{center}
\resizebox{1.0\linewidth}{!}{
  \begin{minipage}{\linewidth}
\begin{align*}
\tilde{\mathfrak{a}}^{r}_{\bullet}W_{\bullet}\mathcal{O}_{Y_{\bullet}}\oplus (W_{\bullet}
& 
\mathcal{O}_{Y_{\bullet}})_{[F]}\xrightarrow{d\oplus d}\tilde{\mathfrak{a}}^{r-1}_{\bullet}W_{\bullet}\Omega_{Y_{\bullet}/(A/p^{\bullet})}^{1}\oplus (W_{\bullet}\Omega_{Y_{\bullet}/(A/p^{\bullet})}^{1})_{[F]}\xrightarrow{d\oplus d}\cdots \\
& \cdots\xrightarrow{d\oplus d}\tilde{\mathfrak{a}}_{\bullet}W_{\bullet}\Omega_{Y_{\bullet}/(A/p^{\bullet})}^{r-1}\oplus (W_{\bullet}\Omega_{Y_{\bullet}/(A/p^{\bullet})}^{r-1})_{[F]}\xrightarrow{d+dV}W_{\bullet}\Omega_{Y_{\bullet}/(A/p^{\bullet})}^{r}\xrightarrow{d}\cdots
\end{align*}
\end{minipage}}
\end{center}
are quasi-isomorphic. This implies that
\begin{equation*}
(\mathcal{P}_{A/R})_{r}=\mathbb{H}^{n}(Y_{\bullet},\mathcal{N}_{rel/R}^{r}W_{\bullet}\Omega_{Y_{\bullet}/(A/p^{\bullet})}^{\bullet})
\end{equation*}
and the divided Frobenius on $(\mathcal{P}_{A/R})_{r}$ is induced by the divided Frobenius on $\mathcal{N}_{rel/R}^{r}W_{\bullet}\Omega_{Y_{\bullet}/(A/p^{\bullet})}^{\bullet}$. It is unique because $A$ and $W(A)$ are $p$-torsion free.
\par
The morphism of frames $\mathcal{W}_{A/R}\xrightarrow{\epsilon}\mathcal{W}_{R}$ induces a base change $\epsilon_{\ast}$ on displays. Let $\tilde{X}:=\tilde{\mathcal{Y}}\times_{\text{Spf }W(A)}\text{Spec }W(R)$ be the induced lifting of $X$ over $W(R)$. The standard display defined on $\tilde{L}_{0}\oplus\cdots\oplus\tilde{L}_{n}$ with
\begin{equation*}
\tilde{L}_{i}=H^{n-i}(\tilde{\mathcal{Y}},\Omega_{\tilde{\mathcal{Y}}/\text{Spf }W(A)}^{i})
\end{equation*}
is transformed into the standard display on $L_{0}\oplus\cdots\oplus L_{n}$ with
\begin{equation*}
L_{i}=\tilde{L}_{i}\otimes_{W(A)}W(R)=H^{n-i}(\tilde{X},\Omega_{\tilde{X}/W(R)}^{i})
\end{equation*}
Note that under the composite map $\kappa:A\rightarrow W(A)\rightarrow W(R)$, $\kappa(a)\in I_{R}$ for $a\in\mathfrak{a}$, hence the image of $\tilde{\mathfrak{a}}\oplus I_{A}$ in $W(R)$ is $I_{R}$.
\par
We have canonical reduction maps
\begin{equation*}
\mathcal{F}il_{rel/R}^{[r]}\Omega_{Y_{\bullet,\bullet}/W_{\bullet}(A/p^{\bullet})}^{\bullet}\rightarrow\mathcal{F}^{r}\Omega_{\tilde{X}/W_{\bullet}(R)}^{\bullet}
\end{equation*}
and
\begin{equation*}
\mathcal{N}_{rel/R}^{r}W_{\bullet}\Omega_{Y_{\bullet}/(A/p^{\bullet})}^{\bullet}\rightarrow\mathcal{N}^{r}W_{\bullet}\Omega_{X/R}^{\bullet}
\end{equation*}
Under the base change of displays $\epsilon_{\ast}\mathcal{P}_{A/R}$ is a display over $R$ with $(\epsilon_{\ast}\mathcal{P}_{A/R})_{r}$ given by the hypercohomology of $\mathcal{F}^{r}\Omega_{\tilde{X}/W(R)}^{\bullet}$. Since the quasi-isomorphisms between $\mathcal{F}il_{rel/R}^{[r]}\Omega_{Y_{\bullet,\bullet}/W_{\bullet}(A/p^{\bullet})}^{\bullet}$ and $\mathcal{N}_{rel/R}^{r}W_{\bullet}\Omega_{Y_{\bullet}/(A/p^{\bullet})}^{\bullet}$ and between $\mathcal{F}^{r}\Omega_{\tilde{X}/W(R)}^{\bullet}$ and $\mathcal{N}^{r}W_{\bullet}\Omega_{X/R}^{\bullet}$ are compatible under the canonical reduction maps, we see that the divided Frobenius $F_{r}$ on $(\epsilon_{\ast}\mathcal{P}_{A/R})_{r}$ obtained by base change coincides with the divided Frobenius on the Nygaard complexes $\mathcal{N}^{r}W_{\bullet}\Omega_{X/R}^{\bullet}$. This finishes the proof of Theorem \ref{Display theorem}(b).
\subsection{Proof of Theorem \ref{Relative displays theorem}(b)}
As in the theorem, we assume that $R$ is an artinian local $W(k)$-algebra with residue field $k$, and that the special fibre $X_{0}$ is a smooth projective variety with smooth versal deformation space $\mathfrak{S}$. Write $\mathfrak{X}/\mathfrak{S}$ for the versal family. Then $\mathfrak{S}\cong\text{Spf }A$, where $A=W(k)\llbracket t_{1},\ldots,t_{h}\rrbracket$ is a formal power series algebra over $W(k)$.
\par
Suppose now that $X_{S}$ is a deformation of $X/R$ over a PD-thickening $S\twoheadrightarrow R$ and let us write $\mathcal{P}_{S}(X_{S})$ for the $\mathcal{W}_{S}$-display structure on $H_{\cris}^{n}(X/W(S))$. Write $u:\mathcal{W}_{S}\rightarrow\mathcal{W}_{S/R}$ for the frame homomorphism. Then we must prove that the relative display $\mathcal{P}_{S/R}=u_{\ast}\mathcal{P}_{S}$ does not depend on the lifting $X_{S}$. That is to say, given another deformation $X_{S}'$ of $X$ over $S$ with associated display $\mathcal{P}_{S}(X_{S}')$, the relative displays $\mathcal{P}_{S/R}(X_{S}):=u_{\ast}\mathcal{P}_{S}(X_{S})$ and $\mathcal{P}_{S/R}(X_{S}'):=u_{\ast}\mathcal{P}_{S}(X_{S}')$ coincide.
\par
By the versality of $\mathfrak{S}$, the deformations $X_{S}$ and $X_{S}'$ are induced by two $W(k)$-algebra homomorphisms $A\overset{x}{\underset{y}{\rightrightarrows}}
 S$. Let $\mathcal{A}_{\text{triv}}=(A,0,A,\sigma,\sigma/p)$ be the trivial frame for $A$ and write $\mathcal{P}_{A}^{\text{triv}}$ for the $\mathcal{A}_{\text{triv}}$-window structure on the versal family (given by \cite{LZ07} Thm 5.5). Then $\mathcal{P}_{S/R}(X_{S})$ and $\mathcal{P}_{S/R}(X_{S}')$ are the base change of $\mathcal{P}_{A}^{\text{triv}}$ along the two induced frame homomorphisms
\begin{equation*}
\mathcal{A}_{\text{triv}}\mathrel{\mathop{\rightrightarrows}^{x}_{y}}\mathcal{W}_{S/R}
\end{equation*} 
that arise from the commutative diagram
\begin{center}
\begin{tikzpicture}
\node (A) at (-1,0) {$A$};
\node (B) at (-1,-1.2) {$A$};
\node (C) at (1,0) {$S$};
\node (D) at (1,-1.2) {$R$};

\path[->,font=\scriptsize] (A) edge node [left] {$=$} (B);
\path[->,font=\scriptsize] (B) edge (D);
\path[->,font=\scriptsize] (C) edge (D);
\path[->,font=\scriptsize,>=angle 90]
([yshift= 2pt]A.east) edge node[above] {$x$} ([yshift= 2pt]C.west)
([yshift= -2pt]A.east) edge node[below] {$y$} ([yshift= -2pt]C.west);

\end{tikzpicture}
\end{center}
That is $\mathcal{P}_{S/R}(X_{S})=x_{\ast}\mathcal{P}_{A}^{\text{triv}}$ and $\mathcal{P}_{S/R}(X_{S}')=y_{\ast}\mathcal{P}_{A}^{\text{triv}}$.
\par 
Now consider the following diagram
\begin{center}
\begin{tikzpicture}[descr/.style={fill=white,inner sep=1.5pt}]
        \matrix (m) [
            matrix of math nodes,
            row sep=2.5em,
            column sep=2.5em,
            text height=1.5ex, text depth=0.25ex
        ]
        { 0 & J
         & B:=A\hat{\otimes}_{W(k)}A & A & 0 \\
           &  & S & R & \\
        };

        \path[overlay,->, font=\scriptsize]
        (m-1-1) edge (m-1-2)
        (m-1-2) edge (m-1-3)
        (m-1-3) edge node [above]{\text{mult.}} (m-1-4)
        (m-1-4) edge (m-1-5)
        (m-1-3) edge (m-2-3)
        (m-1-4) edge (m-2-4);
        
        \path[overlay,->>, font=\scriptsize]
        (m-2-3) edge (m-2-4);
\end{tikzpicture} 
\end{center}
Write $D_{B}(J)$ for the PD-envelope of $(B,J)$. Similarly, set $A_{0}:=W(k)[T_{1},\ldots,T_{h}]$, $B_{0}:=A_{0}\otimes_{W(k)}A_{0}$, $J_{0}:=\ker(B_{0}\xrightarrow{\text{mult.}}A_{0})$ and write $D_{B_{0}}(J_{0})$ for the PD-envelope of $(B_{0},J_{0})$. Then $D_{B_{0}}(J_{0})$ is the PD-polynomial algebra over $B_{0}$ in $h$ variables. Since $A=W(k)\llbracket t_{1},\ldots,t_{h}\rrbracket$ is flat over $A_{0}=W(k)[t_{1},\ldots, t_{h}]$, \cite{BO78} Prop. 3.21 gives that $D_{B}(J)$ is the PD-polynomial algebra over $B$ in $h$ variables. In particular, $D_{B}(J)$ is a flat $D_{B_{0}}(J_{0})$-module, so is certainly $p$-torsion free. We get a diagram
\begin{center}
\begin{tikzpicture}[descr/.style={fill=white,inner sep=1.5pt}]
        \matrix (m) [
            matrix of math nodes,
            row sep=2.5em,
            column sep=2.5em,
            text height=1.5ex, text depth=0.25ex
        ]
        { D_{B}(J) & S \\
          A & R & \\
        };

        \path[overlay,->, font=\scriptsize]
        (m-1-1) edge (m-1-2)
        (m-1-1) edge (m-2-1)
        (m-1-2) edge (m-2-2)
        (m-2-1) edge (m-2-2);

        \end{tikzpicture} 
\end{center}
Let $\widehat{D_{B}(J)}:=\varprojlim_{n} D_{B}(J)/p^{n}$ denote the $p$-adic completion of $D_{B}(J)$. Then $\mathcal{A}=\left(\widehat{D_{B}(J)}\rightarrow A\right)$ is a frame for $A$. The sections $A\rightrightarrows A\otimes_{W(k)}A$ induce frame morphisms 
\begin{equation*}
\mathcal{A}_{\text{triv}}\rightrightarrows\mathcal{A}\rightarrow\mathcal{W}_{S/R}
\end{equation*}
given by the following diagram 
\begin{center}
\begin{tikzpicture}
\node (A) at (-2,0) {$A$};
\node (B) at (-2,-1.5) {$A$};
\node (C) at (0,0) {$\widehat{D_{B}(J)}$};
\node (D) at (0,-1.5) {$A$};
\node (E) at (2,0) {$S$};
\node (F) at (2,-1.5) {$R$};

\path[->,font=\scriptsize] (A) edge (B);
\path[->,font=\scriptsize] (B) edge (D);
\path[->,font=\scriptsize] (C) edge (D);
\path[->,font=\scriptsize] (C) edge (E);
\path[->,font=\scriptsize] (D) edge (F);
\path[->,font=\scriptsize] (E) edge (F);
\path[->,font=\scriptsize,>=angle 90]
([yshift= 2pt]A.east) edge ([yshift= 2pt]C.west)
([yshift= -2pt]A.east) edge ([yshift= -2pt]C.west);

\end{tikzpicture}
\end{center}
Since the geometric construction of windows on $H_{\cris}^{n}(\mathfrak{X}/A)$ (\cite{LZ07}, Thm. 5.5) is compatible with base change (\cite{LZ07}, Cor. 5.6), the base change of $\mathcal{P}_{A}^{\text{triv}}$ along both frame morphisms $\mathcal{A}_{\text{triv}}\rightrightarrows\mathcal{A}$ gives the same $\mathcal{A}$-window; it is the $\mathcal{A}$-window $\mathcal{P}_{A}$  given by applying (\cite{LZ07}, Thm. 5.5) to the frame $\mathcal{A}$. We may now conclude the proof since the $\mathcal{W}_{S/R}$-displays $\mathcal{P}_{S/R}(X_{S})$ and $\mathcal{P}_{S/R}(X'_{S})$ are both given by the base change of $\mathcal{P}_{A}$ along $\mathcal{A}\rightarrow\mathcal{W}_{S/R}$.

\appendix
\section{Appendix}

In this appendix we recall the basic definitions of frames, windows and displays as given in \cite{LZ07} and \cite{LZ19}.

\begin{Def}\label{frame}
Let $R$ be a ring such that $p$ is topologically nilpotent in $R$. A frame $(A,\sigma,\alpha)$ for $R$ consists of a torsion-free $p$-adic ring $A$ with an endomorphism $\sigma:A\rightarrow A$ lifting the Frobenius on $A/p$ and a surjective homomorphism $\alpha:A\rightarrow R$ such that the kernel $\mathfrak{a}=\ker\alpha$ has divided powers. 
\end{Def}

\begin{Def}\label{window}
Let $\mathcal{A}=(A,\sigma,\alpha)$ be a frame for $R$. An $\mathcal{A}$-window consists of 
\begin{enumerate}
\item a finitely generated projective $A$-module $P_{0}$
\item a descending filtration of $P_{0}$ by $A$-submodules 
\begin{equation*}
P_{i+1}\subset P_{i}\subset\cdots\subset P_{1}\subset P_{0}
\end{equation*}
\item $\sigma$-linear homomorphisms 
\begin{equation*}
F_{i}:P_{i}\rightarrow P_{0}
\end{equation*}
\end{enumerate}
such that the following conditions are satisfied
\begin{enumerate}
\item $\mathfrak{a}P_{i}\subset P_{i+1}$ ; $P_{i+1}/\mathfrak{a}P_{i}$ is a finitely generated projective $R$-module $E_{i+1}$ for $i\geq 0$. Let $E_{0}=P_{0}/\mathfrak{a}P_{0}$.
\item The inclusions in 2) induce injective $R$-module homomorphisms
\begin{equation*}
E_{i+1}\rightarrow E_{i}\rightarrow \cdots\rightarrow E_{0}
\end{equation*}
\item $\mathfrak{a}P_{i}=P_{i+1}$ for $i$ large enough.
\item $F_{i}(x)=pF_{i+1}(x)$ for $x\in P_{i+1}$.
\item The union of the images $F_{i}(P_{i})$ for $i\in\mathbb{Z}_{\geq 0}$ generate $P_{0}$ as an $A$-module.
\end{enumerate}
\end{Def}

It is then shown in \cite{LZ07} page 181 that a window is isomorphic to a standard window, that is there are finitely generated projective $A$-modules $L_{0},\ldots, L_{d}$ with $\displaystyle\bigoplus_{i=0}^{d}L_{i}=P_{0}$ and $\sigma$-linear homomorphisms $\Phi_{i}:L_{i}\rightarrow\displaystyle\bigoplus_{j=0}^{d}L_{j}$ such that the determinant of $\Phi_{0}\oplus\cdots\oplus\Phi_{d}$ is a unit. Attached to this data we set for $i\geq 0$ 
\begin{equation*}
P_{i}=\mathfrak{a}^{i}L_{0}\oplus\mathfrak{a}^{i-1}L_{1}\oplus\cdots\oplus\mathfrak{a}L_{i-1}\oplus L_{i}\oplus\cdots\oplus L_{d}
\end{equation*}  
and define $F_{i}$ on $P_{i}$, that is $F_{i}|_{\mathfrak{a}^{i-k}L_{k}}$ for $k<i$ resp. $F_{i}|_{L_{k}}$ for $k\geq i$, as follows:
\begin{align*}
& F_{i}(ax)=\frac{\sigma(a)}{p^{i-k}}\Phi_{k}(x)\text{ for }0\leq k<i, \ x\in L_{k}, \ a\in\mathfrak{a}^{i-k} \\
& F_{i}(x)=p^{k-i}\Phi_{k}(x)\text{ for }i\leq k, \ x\in L_{k}.
\end{align*}
Then these data $(P_{i},F_{i})$ and the obvious inclusions $P_{i+1}\rightarrow P_{i}$ form a window called a standard window.

To define higher displays we will use frames over the ring of Witt vectors for a given $p$-adic ring $R$.

\begin{Def}\label{Witt frame}
Let $S$ be a $p$-adic ring and $W(S)$ its Witt vectors. The Witt frame $\mathcal{W}_{S}=(W(S),\mathcal{J}=I_{S},S,\sigma,\dot{\sigma})$ consists of the data $I_{S}=VW(S)$, $W(S)\rightarrow S$ the augmentation map with kernel $I_{S}$, $\sigma$ the Frobenius on $W(S)$ and $\dot{\sigma}:I_{S}\rightarrow W(S)$, $V\xi\mapsto\xi$.
\end{Def}

\begin{Def}\label{relative Witt frame}
Let $S\rightarrow R$ be a surjective homomorphism of $p$-adic rings such that $\mathfrak{a}$ becomes nilpotent in $S/pS$. Then the relative Witt frame $\mathcal{W}_{S/R}=(W(S),\mathcal{J},R,\sigma,\dot{\sigma})$ consists of the ideal $\mathcal{J}=\ker(W(S)\rightarrow R)$ which is a direct sum $\mathcal{J}=\tilde{\mathfrak{a}}\oplus I_{S}$ where $I_{S}=VW(S)$ is the augmentation ideal in $W(S)$ and $\tilde{\mathfrak{a}}\subset W(S)$ is the ideal consisting of logarithmic Teichm\"{u}ller representatives of elements of $\mathfrak{a}$ (\cite{Zin02}, 1.4). We recall the definition: 

The divided powers on $\mathfrak{a}$ yield divided Witt polynomials 
\begin{equation*}
w'_{n}(\underline{a})=\sum_{i=0}^{n}p^{i}a_{i}^{p^{n}-i}=\sum_{i=0}^{n}(p^{n-i}-1)!\gamma_{p^{n-i}}(a_{i})
\end{equation*}
for $\underline{a}=(a_{0},a_{1},\ldots)\in W(\mathfrak{a})$, and an isomorphism
\begin{align*}
\log : \ & W(\mathfrak{a})\xrightarrow{\simeq}\mathfrak{a}^{\mathbb{N}} \\
& \underline{a}\mapsto(w'_{0}(\underline{a}),\ldots,w'_{n}(\underline{a}),\ldots)
\end{align*}
Define $\tilde{\mathfrak{a}}=\log^{-1}(\mathfrak{a},0,0,\ldots)$. This is an ideal in $W(S)$. 

The map $\sigma$ is the Frobenius on $W(S)$. We have $\sigma(\tilde{\mathfrak{a}})=0$, $I_{S}\cdot\tilde{\mathfrak{a}}=0$ and define $\dot{\sigma}:\mathcal{J}\rightarrow W(S)$ by $\dot{\sigma}(a+V\xi)=\xi$ for $a\in\tilde{\mathfrak{a}}, \ \xi\in W(S)$.
\end{Def}

Then both $\mathcal{W}_{S}$ and $\mathcal{W}_{S/R}$ are equipped with maps called ``Verj\"{u}ngung''. In the case of $\mathcal{W}_{S}$ these are maps $\nu:I_{S}\otimes I_{S}\rightarrow I_{S}$, $V\xi_{1}\otimes V\xi_{2}\mapsto V(\xi_{1}\xi_{2})$ with iterations $\nu^{(k)}:I_{S}^{\otimes k}\rightarrow I_{S}$, $V\xi_{1}\otimes\cdots\otimes V\xi_{k}\mapsto V(\xi_{1}\cdots\xi_{k})$,  and $\pi:I_{S}\rightarrow I_{S}$, $V\xi\mapsto pV\xi$. In the case of $\mathcal{W}_{S/R}$, the Verj\"{u}ngung consists of the two maps $\nu:\mathcal{J}\otimes_{W(S)}\mathcal{J}\rightarrow\mathcal{J}$ and $\pi:\mathcal{J}\rightarrow\mathcal{J}$ with 
\begin{equation*}
\nu((a_{1}+V\xi_{1})\otimes(a_{2}+V\xi_{2}))=a_{1}\cdot a_{2}+V(\xi_{1}\cdot\xi_{2})
\end{equation*}
and 
\begin{equation*}
\pi(a+V\xi)=a+pV\xi
\end{equation*}
with obvious iterations $\nu^{(k)}$ that satisfy the properties (3) and (4) on page 460 of \cite{LZ19}.

In the following let $\mathcal{F}$ be one of the frames considered above, that is $\mathcal{F}=\mathcal{W}_{S}$ or $\mathcal{F}=\mathcal{W}_{S/R}$.

\begin{Def}\label{predisplay}
An $\mathcal{F}$-predisplay $(P_{i},\iota_{i},\alpha_{i},F_{i})$ consists of the following data:
\begin{enumerate}
\item A sequence of $W(S)$-modules $P_{i}$ for $i\geq 0$.
\item Two sets of $W(S)$-module homomorphisms 
\begin{equation*}
\iota_{i}:P_{i+1}\rightarrow P_{i} \ , \ \\ \ \ \alpha_{i}:\mathcal{J}\otimes_{W(S)}P_{i}\rightarrow P_{i+1} 
\end{equation*}
for $i\geq 0$.
\item A set of $\sigma$-linear maps for $i\geq 0$
\begin{equation*}
F_{i}:P_{i}\rightarrow P_{0}
\end{equation*}
\end{enumerate}
which satisfy the following properties:
\begin{enumerate}
\item Consider the following morphisms:
\begin{center}
\begin{tikzpicture}[descr/.style={fill=white,inner sep=1.5pt}]
        \matrix (m) [
            matrix of math nodes,
            row sep=2.5em,
            column sep=2.5em,
            text height=1.5ex, text depth=0.25ex
        ]
        { \mathcal{J}\otimes P_{i} & P_{i+1}  \\
          \mathcal{J}\otimes P_{i-1} & P_{i} & \\
        };

        \path[overlay,->, font=\scriptsize]
        (m-1-1) edge node [above]{$\alpha_{i}$} (m-1-2)
        (m-1-1) edge node [left]{$\mathrm{id}_{\mathcal{J}}\otimes\iota_{i-1}$} (m-2-1)
        (m-1-2) edge node [right]{$\iota_{i}$} (m-2-2)
        (m-2-1) edge node [below]{$\alpha_{i-1}$} (m-2-2);
        
\end{tikzpicture} 
\end{center}
the compositions $\iota_{i}\circ\alpha_{i}$ and $\alpha_{i-1}\circ(\mathrm{id}_{\mathcal{J}}\otimes\iota_{i-1})$ are the multiplication maps $\mathcal{J}\otimes P_{i}\rightarrow P_{i}$ for all $i$.
\item $F_{i+1}\otimes \alpha_{i}=\tilde{F}_{i}$
\end{enumerate}
\end{Def}
Here, for each $\sigma$-linear map $f:M\rightarrow N$ between $W(S)$-modules $M,N$, we define a new $\sigma$-linear map $\tilde{f}:\mathcal{J}\otimes M\rightarrow N$ by $\tilde{f}(\eta\otimes m)=\dot{\sigma}(\eta)f(m)$ for $\eta\in\mathcal{J}$.

In the following we define standard data of a display over a Witt frame. In the case of $\mathcal{W}_{S}$ the definition is given in \cite{LZ07} Definition 2.5. We only recall the definition of standard data of a display over the relative Witt frame $\mathcal{W}_{S/R}$ given in \cite{LZ19} page 460-461.

\begin{Def}\label{standard}
A display given by standard data over the relative Witt frame $\mathcal{W}_{S/R}$ is given as follows:

It consists of finitely generated projective $W(S)$-modules $L_{0},\ldots, L_{d}$ and $\sigma$-linear homomorphisms 
\begin{equation*}
\Phi_{i}:L_{i}\rightarrow L_{0}\oplus\cdots\oplus L_{d}
\end{equation*}
such that $\Phi_{0}\oplus\cdots\oplus\Phi_{d}:L_{0}\oplus\cdots\oplus L_{d}\rightarrow L_{0}\oplus\cdots\oplus L_{d}$ is a $\sigma$-linear isomorphism. Define $\mathcal{J}_{i}=\tilde{\mathfrak{a}}^{i}\oplus VW(S)$. Set
\begin{equation*}
P_{i}=\mathcal{J}_{i}L_{0}\oplus\mathcal{J}_{i-1}L_{1}\oplus\cdots\oplus\mathcal{J}L_{i-1}\oplus L_{i}\oplus\cdots\oplus L_{d}
\end{equation*}
The map $\iota_{i}$ is defined by the following diagram
\begin{center}
\begin{tikzpicture}[descr/.style={fill=white,inner sep=1.5pt}]
        \matrix (m) [
            matrix of math nodes,
            row sep=2.5em,
            column sep=-0.3em,
            text height=1.5ex, text depth=0.25ex
        ]
        { \mathcal{J}_{i+1}L_{0} & \oplus & \mathcal{J}_{i}L_{1} & \oplus & \cdots & \oplus & \mathcal{J}L_{i} & \oplus & L_{i+1} & \oplus & \cdots & \oplus & L_{d} \\
            \mathcal{J}_{i}L_{0} & \oplus & \mathcal{J}_{i-1}L_{1} & \oplus & \cdots & \oplus & L_{i} & \oplus & L_{i+1} & \oplus & \cdots & \oplus & L_{d}\\
        };

        \path[overlay,->, font=\scriptsize]
        (m-1-1) edge node [left]{$\pi$} (m-2-1)
        (m-1-3) edge node [left]{$\pi$} (m-2-3)
        (m-1-7) edge node [left]{$\mathrm{id}$} (m-2-7)
        (m-1-9) edge node [left]{$\mathrm{id}$} (m-2-9)
        (m-1-13) edge node [left]{$\mathrm{id}$} (m-2-13);
        
\end{tikzpicture}
\end{center}
The homomorphisms $\alpha_{i}:\mathcal{J}\otimes P_{i}\rightarrow P_{i+1}$ are defined as follows
\begin{center}
\begin{tikzpicture}[descr/.style={fill=white,inner sep=1.5pt}]
        \matrix (m) [
            matrix of math nodes,
            row sep=2.5em,
            column sep=-0.3em,
            text height=1.5ex, text depth=0.25ex
        ]
        { \mathcal{J}\otimes\mathcal{J}_{i}L_{0} & \oplus & \mathcal{J}\otimes\mathcal{J}_{i-1}L_{1} & \oplus & \cdots & \oplus & \mathcal{J}\otimes L_{i} & \oplus & \mathcal{J}\otimes L_{i+1} & \oplus & \cdots & \oplus & \mathcal{J}\otimes L_{d} \\
            \mathcal{J}_{i+1}L_{0} & \oplus & \mathcal{J}_{i}L_{1} & \oplus & \cdots & \oplus & L_{i} & \oplus & L_{i+1} & \oplus & \cdots & \oplus & L_{d}\\
        };

        \path[overlay,->, font=\scriptsize]
        (m-1-1) edge node [left]{$\nu$} (m-2-1)
        (m-1-3) edge node [left]{$\nu$} (m-2-3)
        (m-1-7) edge node [left]{$\mathrm{mult}$} (m-2-7)
        (m-1-9) edge node [left]{$\mathrm{mult}$} (m-2-9)
        (m-1-13) edge node [left]{$\mathrm{mult}$} (m-2-13);
        
\end{tikzpicture}
\end{center}
Finally we define $\sigma$-linear maps $F_{i}:P_{i}\rightarrow P_{0}$ as 
\begin{center}
\begin{tikzpicture}[descr/.style={fill=white,inner sep=1.5pt}]
        \matrix (m) [
            matrix of math nodes,
            row sep=2.5em,
            column sep=-0.3em,
            text height=1.5ex, text depth=0.25ex
        ]
        { \mathcal{J}_{i}L_{0} & \oplus & \cdots & \oplus & \mathcal{J}L_{i-1} & \oplus & L_{i} & \oplus & L_{i+1} & \oplus & L_{i+2} & \oplus & \cdots \\
           L_{0} & \oplus & \cdots & \oplus & L_{i-1} & \oplus & L_{i} & \oplus & L_{i+1} & \oplus & L_{i+2} & \oplus & \cdots\\
        };

        \path[overlay,->, font=\scriptsize]
        (m-1-1) edge node [left]{$\tilde{\Phi}_{0}$} (m-2-1)
        (m-1-5) edge node [left]{$\tilde{\Phi}_{i-1}$} (m-2-5)
        (m-1-7) edge node [left]{$\Phi_{i}$} (m-2-7)
        (m-1-9) edge node [right]{$p\Phi_{i+1}$} (m-2-9)
        (m-1-11) edge node [right]{$p^{2}\Phi_{i+2}$} (m-2-11);
        
\end{tikzpicture}
\end{center}
where $\tilde{\Phi}_{j}$ is defined by $\tilde{\Phi}_{j}(\eta l_{j})=\dot{\sigma}(\eta)\Phi_{j}(l_{j})$ for $\eta\in\mathcal{J}_{j}$, $l_{j}\in L_{j}$, $j<i$.
\end{Def}
These data meet the requirements of a predisplay.

\begin{Def}\label{display}
Let $\mathcal{F}$ be either of the Witt frames considered above. Then an $\mathcal{F}$-display is an $\mathcal{F}$-predisplay which is isomorphic to the display associated to standard data. The choice of such an isomorphism is called a normal decomposition. 
\end{Def}

For $S\rightarrow R$ a surjective PD-morphism one has the following morphisms of frames equipped with Verj\"{u}ngung
\begin{equation*}
\mathcal{W}_{S}\xrightarrow{\epsilon}\mathcal{W}_{S/R}\rightarrow\mathcal{W}_{R}
\end{equation*}

\begin{Remark}\label{base change}
 For a morphism of frames with Verj\"{u}ngung $u:\mathcal{F}\rightarrow\mathcal{F}'$, we have a base change map of displays, that is the $\mathcal{F}'$-display $u_{\ast}\mathcal{P}$ obtained by base change of an $\mathcal{F}$-display $\mathcal{P}$ exists (\cite{LZ19}, Prop. 6).
\end{Remark}

For any $\mathcal{W}_{S}$-display $\mathcal{P}_{S}$ we can associate the base change $\mathcal{P}_{S/R}=\epsilon_{\ast}\mathcal{P}_{S}$, we also call the relative display for the morphism $S\rightarrow R$ associated to $\mathcal{P}_{S}$. It is clear from the definitions that if $\mathcal{P}_{S}$ has a normal decomposition with finitely generated projective $W(S)$-modules $L_{i}$, $i=0,\ldots, d$, then using the same finitely generated projective modules $L_{i}$ for the relative display $\mathcal{P}_{S/R}$, the inclusion map $I_{S}=VW(S)\rightarrow\tilde{\mathfrak{a}}^{i}+I_{S}=\mathcal{J}_{i}$ (for $\mathfrak{a}=\ker(S\rightarrow R)$) and the obvious extensions of $\iota_{i}$, $\alpha_{i}$, $F_{i}$ to the $P_{i}$ in Definition \ref{standard} built from the $L_{i}$ defines the standard data  for $u_{\ast}\mathcal{P}_{S}=:\mathcal{P}_{S/R}$.

\begin{Remark}
If $(P_{i})$ is an $\mathcal{A}$-window for a frame $\mathcal{A}=(A\rightarrow R)$ as in definitions \ref{frame}--\ref{window}, then there is an induced $\mathcal{W}_{R}$-display given by the composite ring homomorphism $\kappa:A\rightarrow W(A)\rightarrow W(R)$ that satisfies
\begin{align*}
& \kappa(\sigma(a))=F\kappa(a) \ \ \ ; \ a\in A \\
& \kappa\left(\frac{\sigma(a)}{p}\right)=V^{-1}\kappa(a) \ \ \ ; \ a\in\mathfrak{a}
\end{align*}
It is described explicitly for standard data associated to the window and also gives an invariant construction in \cite{LZ07} page 182.
\end{Remark}

\begin{Remark}
For $S\rightarrow R$ as above and the induced morphism of frames $\mathcal{W}_{S}\rightarrow\mathcal{W}_{R}$, the base change from $\mathcal{W}_{S}$-displays to $\mathcal{W}_{R}$-displays coincides with the one given in \cite{LZ07} Prop. 2.12.
\end{Remark}

\
\\
\
\\
\noindent 
Oli Gregory \hfill Andreas Langer \\
Technische Universit\"{a}t M\"{u}nchen \hfill University of Exeter \\
Zentrum Mathematik - M11 \hfill Mathematics \\
Boltzmannstra{\ss}e 3 \hfill Exeter EX4 4QF \\
85748 Garching bei M\"{u}nchen, Germany \hfill Devon, UK \\
email: oli.gregory@tum.de \hfill email: a.langer@exeter.ac.uk

\end{document}